\documentclass[11pt]{article}
\usepackage{amsmath, amssymb, amscd, amsthm, amsfonts}
\usepackage{graphicx}
\usepackage{hyperref}
\usepackage{color}
\usepackage{dsfont}

\def\be{\begin{eqnarray}}
\def\ee{\end{eqnarray}}
\def\b*{\begin{eqnarray*}}
\def\e*{\end{eqnarray*}}

\newcommand{\rmi}{{\rm (i)$\>\>$}}
\newcommand{\rmii}{{\rm (ii)$\>\>$}}
\newcommand{\rmiii}{{\rm (iii)$\>\>$}}

\def \E{\mathbb{E}}
\def \F{\mathbb{F}}

\def \M{\mathbb{M}}
\def \N{\mathbb{N}}

\def \P{\mathbb{P}}
\def \Q{\mathbb{Q}}
\def \R{\mathbb{R}}

\def\Cc{{\cal C}}
\def\Dc{{\cal D}}

\def\Fc{{\cal F}}

\def\Lc{{\cal L}}

\def\Nc{{\cal N}}

\def\Pc{{\cal P}}

\def\Sc{{\cal S}}

\def \de {\delta}
\def \eps {\varepsilon}

\def \dX{\underline{X}}

\def \dL{\underline{L}}
\def \dLa{\underline{\Lambda}}

\def \Pl{\Pc_{\le}}
\def \hr{\hat{d}}
\def \dS{\underline{S}}

\def \bin{{\rm Bin}}

\def \sb {\subset}

\def \d{\mathrm{d}}

\def \0{\bf 0}

\title{Systemic robustness: a mean-field particle system approach}
\author{Erhan Bayraktar\footnote{Department of Mathematics, University of Michigan. Email: erhan@umich.edu. E. Bayraktar is partially supported by the National Science Foundation under grant DMS-2106556 and by the Susan M. Smith chair.} \and Gaoyue Guo\footnote{Université Paris-Saclay, CentraleSupélec, MICS and CNRS FR-3487. Email: gaoyue.guo@centralesupelec.fr.} \and Wenpin Tang\footnote{Department of Industrial Engineering and Operations Research, Columbia University. Email: wt2319@columbia.edu. W. Tang gratefully acknowledges financial support through NSF grants DMS-2113779 and DMS-2206038, and through a start-up grant at Columbia University.} \and Yuming Paul Zhang\footnote{Department of Mathematics, Auburn University. Email: paulzhangyuming@gmail.com.}}
\date{\today}

\newtheorem{theorem}{Theorem}
\newtheorem{lemma}[theorem]{Lemma}
\newtheorem{proposition}[theorem]{Proposition}

\newtheorem{assumption}[theorem]{Assumption}

\begin{document}

\maketitle

\begin{abstract}
This paper is concerned with the problem of budget control in a large particle system 
modeled by stochastic differential equations involving hitting times,
which arises from considerations of systemic risk in a regional financial network. 
Motivated by Tang and Tsai \cite{{TT18}} (Ann. Probab., 46 (2018), pp. 1597--1650),
we focus on the number or proportion of surviving entities that never default to measure the systemic robustness. 
First we show that both the mean-field particle system and its limit McKean-Vlasov equation are well-posed 
by virtue of the notion of minimal solutions.
We then establish a connection between the proportion of surviving entities in the large particle system
and the probability of default in the McKean-Vlasov equation
as the size of the interacting particle system $N$ tends to infinity. 
Finally, we study the asymptotic efficiency of budget control in different economy regimes:
the expected number of surviving entities is of constant order in a negative economy;
it is of order $\sqrt{N}$ in a neutral economy;
and it is of order $N$ in a positive economy where the budget's effect is negligible. 
\end{abstract}

\textit{Key words:} Budget control, drifted Brownian motion, hitting times, interacting particle systems, large system limits, McKean-Vlasov equation, mean-field interactions, systemic risk. 

\medskip
\textit{AMS 2010 Mathematics Subject Classification:} 35K10, 60H30, 60K35, 91G80.

\section{Introduction}
\label{sec:intro}

\quad Systemic risk in financial markets has emerged as a major research topic since the 2008 financial crisis, see e.g. \cite{FL13}.
It refers to the risk that small losses and defaults can escalate to cause an event affecting large parts of the financial industry. 
The existing literature on modeling systemic risk relies on network models \cite{ACM13, AMS13, BGGGS12, CLY16, Glass16, LBS09, MA10}, quantitative risk measures \cite{APPR17, CIM13}, and particle systems \cite{CFS15, DIRT2015a, DIRT2015b, GSSS15, HLS2019, NS2019, DNS2019, CRS2020, CCS2020}. 

This paper aims to analyze systemic risk in a large financial network and adopts the particle system approach with mean-field interactions which is intimately to Stefan's problem (and its variants), see e.g. \cite{ HLS2019, NS2019, CRS2020, BGTZ20, CCS2020}.

Consider an interconnected financial system of $N$ homogeneous banks,
and let $X^{N,1}_t, \ldots, X^{N,N}_t$ be their respective capital levels at time $t \ge 0$.
A bank becomes insolvent if its capital drops below some threshold, 
which is set to be zero without loss of generality.
In the absence of defaults, 
the capital processes $(X_t^{N,i}, \, t \ge 0)$ are supposed to follow
the stochastic differential equations (SDEs): for $i=1,\ldots, N$,
\be
\label{eq:org}
d X^{N,i}_t =  \beta \, d t + d B^i_t, \quad t \ge 0,
\ee
where $\beta\in\R$ stands for the regional economy (see e.g. \cite{HG20}),
and $B^1, \ldots, B^N$ are independent Brownian motions. 

If $k\ge 1$ banks default simultaneously at some time,
we assume that the capital of each remaining bank suffers an immediate loss equal to $G(k/N)$,
which may cause further defaults. 
Here $G:[-1,1]\to\R$\footnote{The existing research focuses on the case $\beta=0$ and $G(x)=\alpha x$ for $\alpha>0$, while \cite{NS2019} considers the alternative case $G(x)=-\alpha\log(1-x)$ for $\alpha > 0$.} is  a continuous and non-decreasing function such that  $G(0)=0$.
After the default event, the remaining capital processes continue to follow the SDEs \eqref{eq:org}
until one of them hits zero, and so on. 
Rephrasing in mathematical language, the capitals evolve as follows:
for $i=1,\ldots, N$,
\be
\label{dyn:particle}
X^{N,i}_t=Z^{N,i}+\beta t +B^i_t - G(L^N_t) &\mbox{and}& L^N_t := \frac{1}{N}\sum_{i=1}^N\mathds{1}_{\{\tau^N_i\le t\}},\quad t\ge 0, 
\ee
where $Z^{N,1},\ldots, Z^{N,N}$ are the initial capitals as i.i.d. random variables that are independent of Brownian motions $B^1,\ldots, B^N$,
and $\tau^N_i := \inf\{t\ge 0: X^{N,i}_t\le 0\}$ is the default time of bank $i$.
Here $L_t^N$ denotes the fraction of insolvent banks up to and including time $t$\footnote{From the practical viewpoint, we are interested in the evolution of $X^{N,i}$ only until time $\tau^N_i$, i.e. $(X^{N,i}_t, \, 0 \le t \le \tau^N_i)$, while we let the process $X^{N,i}$  diffuse freely after its default time for technical purposes.}.
Letting $N\to\infty$, we get formally the McKean-Vlasov limit as follows:
\be
\label{dyn:lim}
 X_t=Z+\beta t +B_t -  G(\Lambda_t) &\mbox{and}& \Lambda_t :=\P(\tau\le t),\quad t\ge 0, 
 \ee
where $Z$ is a random variable (distributed as $Z^{N,1}$),
$B$ is an independent Brownian motion,
and $\tau := \inf\{t\ge 0: X_t\le 0\}$. 
It is known that the case $G(x)=\alpha x$ for $\alpha>0$ provides a probabilistic interpretation
of the supercooled Stefan problem and has given rise to a rich body of literature,
see Delarue et al. \cite{DIRT2015a, DIRT2015b, DNS2019}, Hambly et al. \cite{HLS2019},
Cuchiero et al. \cite{CRS2020, CCS2020}, and Bayraktar et al. \cite{BGTZ20}.

The system transitions between two regimes: 
\begin{itemize}
\item
the {\em well-behaved regime}, in which the system spends most of its time,
and during which the default cascades are very small or do not appear at all;
\item
the {\em systemic crisis regime}, which occurs rarely, 
and which is characterized by a large group of banks defaulting in a short period of time. 
\end{itemize}
In our setting, the \emph{systemic failure} occurs at $t_{sys}:=\{t\ge 0: \Delta \Lambda_t:=\Lambda_t-\Lambda_{t-} >0\}$
when a non-negligible fraction of defaults occur in a short period 
and the system passes abruptly from the well-behaved regime to the systemic crisis regime. 
The systemic event has gained significant attention in the wake of financial crisis (see e.g. \cite{DIRT2015a, NS2019}), 
while we focus on the \emph{systemic robustness}, i.e. the proportion of  banks surviving forever. 
For $t\ge 0$, we denote by
\b*
S^{N}_t :=\sum_{i=1}^N\mathds{1}_{\{\tau^N_i>t\}},
\e*
the number of banks that survive until time $t$. 
Thus, $S^{N}_{\infty}$ (number of banks surviving forever) 
or $1-L^{N}_{\infty}$ (proportion of banks surviving forever) 
encodes the system’s robustness.

Our first goal is to study the expected surviving proportion $\E[S^{N}_{\infty}]/N$, 
which boils down to showing under suitable conditions,
\be
\label{eq:lim_frac}
\lim_{N\to\infty} \frac{\E[S^{N}_{\infty}]}{N} =\P(\tau=\infty).
 \ee
For $\beta\le 0$, the identity \eqref{eq:lim_frac} follows from a simple observation. 
As Brownian motion hits each level in finite time with probability one, 
we have for all $t\ge 0$, 
\b*
   X^{N,i}_t\le Z^{N,i} +B^i_t &\mbox{and}&
 X_t\le Z +B_t,  
 \e*
which implies that 
$\mathds{1}_{\{\tau^N_i=\infty\}} \le \mathds{1}_{\{Z^{N,i} +B^i_t>0, \, \forall t\ge 0\}}=0$ 
and  $\mathds{1}_{\{\tau=\infty\}} \le \mathds{1}_{\{Z+B_t>0, \, \forall t\ge 0\}}=0$ with probability one. 
Hence, $\E [S^{N}_{\infty}] = 0 =\P(\tau=\infty)$, and \eqref{eq:lim_frac} holds. 
The case $\beta>0$ is less obvious, and will be treated in Section \ref{sc3}.
Note that both the particle system \eqref{dyn:particle} and the McKean-Vlasov equation \eqref{dyn:lim} 
may have more than one solution (see e.g. \cite[Section 3.1]{DIRT2015b}),
which renders the statement of our desired result \eqref{eq:lim_frac} more subtle.
Among the solutions to \eqref{dyn:particle} and  \eqref{dyn:lim}, we need to choose suitable ones. 
To this end, we adopt the \emph{minimal solution} as in \cite{CRS2020}, 
and more details will be provided in Section \ref{sc22}.

The next problem is to find a feasible way to compute the minimal solution to the McKean-Vlasov equation \eqref{dyn:lim}.
The idea is to approximate it by a sequence of regularized problems. 
Note that 
\b*
\exp\left(-\frac{1}{\eps}\int_0^t (X^{\eps}_s)^-\d s\right)\stackrel{\eps\to 0}{\longrightarrow} \mathds{1}_{\{\tau>t\}} &\Longrightarrow & \E\left[\exp\left(-\frac{1}{\eps}\int_0^t (X^{\eps}_s)^-\d s\right)\right] \stackrel{\eps\to 0}{\longrightarrow}  \Lambda_t.
\e*
For $\eps>0$, we consider the $\varepsilon$-regularized equation:
\be \label{dyn_reg}
X^{\eps}_t= Z + \beta t + B_t  - G\left( 1- \E\left[\exp\left(-\frac{1}{\eps}\int_0^t (X^{\eps}_s)^-\d s\right)\right]\right), \quad t\ge 0.
\ee  
We show in Theorem \ref{thm:reg} that the equation \eqref{dyn_reg} has a unique solution that converges to the minimal solution 
to the McKean-Vlasov equation \eqref{dyn:lim}. 

Finally, we consider how gouvernemental efforts may steady the banking system. Inspired by Amini, Minca and Sulem \cite{AMS13}, see also \cite{CLY16}, where \emph{equity injections} are applied to weaken the propagation of defaults in graph setting, we adopt the approach of particle systems, and analyse how equity injections may improve the number/proportion of banks surviving forever. By considering the corresponding  control problems, we investigate the budget's significance as $N\to\infty$. It is worth noting that, during the preparation of our manuscript, we learnt that Cuchiero, Reisinger and Rigger built in \cite{CCS2020} an alternative control problem on the same underlying framework. More precisely, our paper aims at maximising the expected number/proportion of banks that never default overall injection strategies subject to given resources, and their focus is to find out the minimal budget ensuring a given expected rate of banks surviving up to some finite time $T>0$, where they assume that all the banks have the same injection strategy.

Granted a budget that is normalized to one unit without loss of generality, 
the regulators aim to divide and allocate it among all surviving banks 
in order to maximize the expected number/proportion of banks that survive forever.
An $\F-$progressively measurable process $\phi \equiv  \big((\phi^1_t,\ldots,\phi^N_t), \, t \ge 0 \big)$ taking values in $[0,1]^N$ is called a strategy if 
\b*
\sum_{i=1}^N\phi^i_t \le 1,\quad  t\ge 0,
\e*
where $\F \equiv  \big(\Fc_t:=\sigma(B^1_s,\ldots, B^N_s,  \, s\in [0, t]), \, t \ge 0 \big)$. 
Here $\phi^i_t$ stands for the equity injected to bank $i$ at time $t$, 
and let $\Phi_N$  be the collection of all strategies.
So the capital processes subject to the strategy $\phi \in \Phi_N$ evolve as follows: for $i=1,\ldots, N$
\be\label{dyn:allo_particle}
   X^{\phi, N,i}_t=Z^{N,i} +\int_0^t (\beta+\phi^i_s)\d s +B^i_t - G(L^{\phi,N}_t) &\mbox{and}& L^{\phi,N}_t := \frac{1}{N}\sum_{i=1}^N\mathds{1}_{\{\tau^{\phi,N}_i\le t\}},\quad  t\ge 0, 
\ee
with $\tau^{\phi,N}_i := \inf\{t\ge 0: X^{\phi,N,i}_t\le 0\}$. Accordingly, the number of banks surviving up to time $t$ is 
\b*
S^{\phi,N}_t := \sum_{i=1}^N\mathds{1}_{\{\tau^{\phi,N}_i>t\}}.
\e*
Our goal is to examine the budget's significance in the large $N$ limit.
To be more precise, we show in Theorem \ref{thm:allo} that $\sup_{\phi \in \Phi_N} \E[S_\infty^{\phi, N}]$ 
has different scaling in $N$ depending on whether the regional economy $\beta < 0$, $\beta = 0$ or $\beta > 0$.

\paragraph{A word on notations.}
For ease of presentation, we drop the superscript $N$ when the context is clear, i.e.
\b*
Z^i\equiv Z^{N,i},~ X^i\equiv X^{N,i},~  \tau_i\equiv \tau^{N}_i,~ \mbox{etc.}
\e*
Furthermore, we write when there is no risk of confusion
\b*
X^i\equiv X^{\phi,i}\equiv X^{\phi, N,i},~ \tau_i\equiv \tau^{\phi}_i\equiv \tau^{\phi, N}_i,~ L^{\phi}\equiv L^{\phi,N},~ \mbox{etc.}
\e*

\medskip
The rest of the paper is organized as follows.
We state the main results in Section \ref{sec:result},
and the proofs are given in Sections \ref{sc3}--\ref{sc5}.

 \section{Preliminaries and main results}
 \label{sec:result}
 
\quad We provide measure theoretical background in Section \ref{ssec:pre}, 
and then present the main results in Section \ref{sc22}.
 
 \subsection{Preliminaries}
 \label{ssec:pre}

\quad Given a generic Polish space $E$,  let $\Pc(E)$ (resp. $\Pl(E)$) be the set of probability (sub-probability) measures on $E$. 
We set $\Pc:=\Pc(\R_+)$ and $\Pl:=\Pl(\R_+)$.   
$\Cc$ (resp. $\Dc$) stands for the space of continuous (resp. c\`adl\`ag) functions on $[-1,\infty)$. Let $\Cc$ be endowed with the compact convergence induced by the metric $d_u$, i.e. 
\b*
d_u(f,f'):=\sum_{n\ge 1}\frac{1}{2^n}\big(\|f-f'\|_{n}\wedge 1\big)\quad \mbox{with } \|f-f'\|_{n}:=\sup_{t\in [0,n]}|f_t-f'_t|.
\e*
As for $\Dc$, we use the Skorokhod $M_1$ topology which turns out to be more convenient for our purpose. 
Denote by $d_m$ the corresponding metric
(see e.g. Whitt \cite[Chapter 12]{Whitt2002} for a complete overview). 
Moreover, we introduce the set
\b*
\M := \big\{\ell\in\Dc    \mbox{ is non-decreasing:}\quad   \lim_{t\to\infty}\ell_t=:\ell_{\infty}\le 1 \mbox{ and } \ell_{t}=0, ~ \forall t\in [-1,0)\big\}
\e* 
and the  L\'evy distance $d$ defined by
\b*
d(\ell,\ell') := \inf\big\{\eps>0:~ \ell_{t+\eps}+\eps \ge \ell'_t\ge \ell_{t-\eps}-\eps,~ \forall t\ge 0\big\},\quad \ell,\ell'\in\M.
\e*
For each $t\ge 0$, define $d_t: \M\times \M\to \R_+$ by $d_t(\ell,\ell'):=d(\ell_{t\wedge},\ell'_{t\wedge})$, and further $\hr$ by
\b*
\hr(\ell,\ell') := \int_{0}^{\infty}e^{-t}\big(d_t(\ell,\ell')\wedge 1\big)\d t,
\e*
where $\ell_{t\wedge}\in\M$ is the function stopped at $t$, i.e. $\ell_{t\wedge s}:=\ell_{\min(t,s)}$ for all $s\ge -1$. 
Lemma \ref{lem:metrics} summarizes the properties of $d$ and $\hr$, and its proof can be read from \cite{Whitt2002}. 
 \begin{lemma}\label{lem:metrics}
 For every $\ell\in \M$, there exists a unique $\theta_{\ell}\in \Pl$ such that $\ell_t=\theta_{\ell}((-\infty,t])$ for all $t\ge 0$. For any $(\ell^n)_{n\ge 1}\sb \M$ and $\ell\in \M$, we have $(1)\Longleftrightarrow (2) \Longrightarrow (3)\Longleftrightarrow (4)$, where
\begin{itemize}
\item[\rm (1)] $\lim_{n\to\infty}d(\ell^n,\ell)=0$;
\item[\rm (2)] $\theta_{\ell^n}$ converges weakly  to $\theta_{\ell}$ in $\Pl$;
\item[\rm (3)] $\lim_{n\to\infty}\hr(\ell^n,\ell)=0$;
\item[\rm (4)]  $\lim_{n\to\infty}\ell^n_t=\ell_t$ for all points of continuity $t\ge 0$ of $\ell$.
\end{itemize}
We have $(\M,d)$ and $(\M,\hr)$ are Polish spaces. 
Moreover, $(\M,\hr)$ is compact.
\end{lemma}

\subsection{Main results}
 \label{sc22}
 
\quad We start with a few vocabularies. 
Given independent random variables $Z^{N,1},\ldots, Z^{N,N}$ and independent Brownian motions $B^1,\ldots, B^N$,  $(X^{N,1},\ldots, X^{N,N}, L^N)$ is said to be a solution to \eqref{dyn:particle} if 
$L^N$ is $\F-$adapted and takes values in $\M$,
and the equations of \eqref{dyn:particle} are satisfied with probability one.
Similarly, for a random variable $Z$ and an independent  Brownian motion $B$, the pair $(X, \Lambda)$ is said to be a solution to \eqref{dyn:lim} if $\Lambda\in\M$ and the equations of  \eqref{dyn:lim} hold. 
Note that $(X^{N,1},\ldots, X^{N,N})$ (resp. $X$) is fully determined if $L^N$ (resp. $\Lambda$) is fixed. 
Hence, we also say that $L^N$ (resp. $\Lambda$) is a solution to \eqref{dyn:particle} (resp. \eqref{dyn:lim}) for the sake of simplicity.  

As previously mentioned,  neither \eqref{dyn:particle} nor \eqref{dyn:lim} can guarantee the uniqueness of their solution. 
Thus, we introduce the minimal solution as follows. 
A solution $(\dX^{N,1},\ldots, \dX^{N,N}, \dL^N)$  (resp. $(\dX, \dLa)$) to \eqref{dyn:particle} (resp. \eqref{dyn:lim}) is called minimal if
\b* 
\dL^N_t\le L^N_t \quad \left(\mbox{resp. } \dLa_t\le \Lambda_t\right),\quad t\ge 0,
\e* 
holds for any solution $(X^{N,1},\ldots, X^{N,N}, L^N)$ (resp. $(X, \Lambda)$) to \eqref{dyn:particle} (resp.  \eqref{dyn:lim}). 
The minimal solutions, if exist, must be unique. Theorem \ref{thm:existence} ensures the existence of minimal solutions, 
which yields the wellposedness of \eqref{dyn:particle} and \eqref{dyn:lim}.
Its proof is given in Section \ref{sc3}.
\begin{theorem}\label{thm:existence}
Both \eqref{dyn:particle} and \eqref{dyn:lim} have a unique minimal solution. 
If $\beta>0$, we have $\dLa_{\infty}<\Lambda_{\infty}$ for any solution $(X, \Lambda)$ to \eqref{dyn:lim} that is different from the minimal solution.
\end{theorem}

From the financial perspective, all banks will default in finite time when $\beta\le 0$.
If $\beta>0$, then $1-\dLa_{\infty}$ is the largest proportion of banks surviving forever in our setting \eqref{dyn:lim}. 

Next we turn to the asymptotic proportion of banks that survive forever. 
As $G$ is continuous, it is uniformly continuous restricted on any compact subset. Let $\alpha: [0,1]\to \R_+$ be the modulus of continuity of $G$, i.e. $|G(z)-G(z')|\le \alpha(|z-z'|)$ for all $z,z'\in [0,1]$. 
We also need the following technical assumption on the initial distributions.

\begin{assumption}\label{ass1}
$Z^{N,1},\ldots, Z^{N,N}$ are i.i.d. random variables 
whose probability distribution $\theta^N\in \Pc$ is such that $(\theta^N, \, N \ge 1)$ converges weakly to $\theta$. 
\end{assumption}

For any $b\in\R$, denote by $Q_{b}:\R\to\R$ the translation, i.e.  $Q_b(z):=z+b$. 
Theorem \ref{thm:limit} relates the expected surviving proportion $\E [S^N_\infty]/N$ to the minimal solution $\dLa$ to the McKean-Vlasov equation \eqref{dyn:lim}.
Its proof is given in Section \ref{sc41}.

\begin{theorem}\label{thm:limit}
Assume that there exists $(\gamma_N, \, N \ge 1) \subset \R_+$ such that
\b*
\lim_{N\to\infty}\gamma_N = \lim_{N\to\infty}\frac{1}{N\gamma_N^2} = 0 &\mbox{and}&
\theta^N = Q_{\alpha(\gamma_N)}(\theta),\quad N\ge 1, 
\e*
where $Q_b(\theta)$ is the image measure of $\theta$ by  $Q_b$. 
Taking the minimal solution $(X^{N,1},\ldots, X^{N,N}, L^N)=(\dX^{N,1},\ldots, \dX^{N,N}, \dL^N)$, it holds that  
\b*
\lim_{N\to\infty}  \E\left[\left|\frac{\dS^N_{\infty}}{N}-(1-\dLa_{\infty})\right|\right]=0.
\e*
\end{theorem}

Theorem \ref{thm:reg} shows that the unique solution to the $\varepsilon$-regularized problem \eqref{dyn_reg}
converges to the minimal solution $\dLa$ to \eqref{dyn:lim}.
Its proof is given in Section \ref{sc42}.
\begin{theorem}\label{thm:reg}
Let the assumptions in Theorem \ref{thm:limit} hold,
and assume further that $G$ is Lipschitz. 
For every $\eps>0$, 
 \eqref{dyn_reg} has a unique solution $X^\eps$ such that $X^\eps$ has continuous paths and $\lim_{\eps\to 0} X^\eps_t =\dX_t$ for all $t\ge 0$.
\end{theorem}

Finally, we consider the efficiency of equity injections. 
Similar to \eqref{dyn:particle}, the equations \eqref{dyn:allo_particle} do not ensure the uniqueness of their solution in general. 
Hence, we adopt the minimal solution as follows.  
$(X^{\phi, N,1},\ldots, X^{\phi, N,N}, L^{\phi,N})$ is said to be a solution to \eqref{dyn:allo_particle}  
if $L^{\phi,N}$ is $\F-$adapted and takes values in $\M$, and the equations of \eqref{dyn:allo_particle} are satisfied with probability one.
We say that $(\dX^{\phi, N,1},\ldots, \dX^{\phi, N,N}, \dL^{\phi,N})$ is a minimal solution to \eqref{dyn:allo_particle} if
\b* 
\dL^{\phi,N}_t\le L^{\phi,N}_t, \quad t\ge 0,
\e* 
holds for any solution $(X^{\phi, N,1},\ldots, X^{\phi, N,N}, L^{\phi,N})$ to \eqref{dyn:allo_particle}.  
\begin{theorem}
\label{thm:allo}
For all $N\ge 1$ and $\phi\in\Phi_N$,  \eqref{dyn:allo_particle} have a unique minimal solution. 
Assume further that $G(x)=\alpha x$ for some $\alpha\in \R_+$.
\begin{enumerate}
\item[(i)]
If $\beta<0$, then the budget is useless and the number of banks surviving forever is finite, i.e.
\b*
\sup_{\phi\in\Phi_N} \E[S^{\phi,N}_{\infty}]<-2/\beta, \quad \mbox{for all } N\ge 1.
\e*
Here we emphasize that $S^{\phi,N}_{\infty}$ corresponds to any solution to \eqref{dyn:allo_particle}. 
\item[(ii)]
If $\beta=0$ and $\theta$ is compactly supported, then the budget allows to maintain the number of banks surviving forever of order $\sqrt N$, i.e.
\b*
0<\liminf_{N\to\infty} \left(\sup_{\phi\in\Phi_N} \E\left[\frac{\dS^{\phi,N}_{\infty}}{\sqrt N}\right] \right) \le \limsup_{N\to\infty} \left(\sup_{\phi\in\Phi_N} \E\left[\frac{\dS^{\phi,N}_{\infty}}{\sqrt N}\right] \right) <\infty.
\e*
\item[(iii)]
If $\beta>0$ and $\alpha=0$, then the budget's effect is negligible and the proportion of banks surviving forever remains unchanged, i.e. 
\b*
\lim_{N\to\infty} \left(\sup_{\phi\in\Phi_N} \E\left[\frac{S^{\phi,N}_{\infty}}{N}\right] \right) = \lim_{N\to\infty} \E\left[\frac{S^{{\bf 0},N}_{\infty}}{N}\right]>0,
\e*
where ${\bf 0}\in\Phi_N$ stands for the strategy $  \big((\phi^1_t\equiv 0,\ldots,\phi^N_t\equiv 0), \, t \ge 0\big)$. 
\end{enumerate}
\end{theorem}
The proof of Theorem \ref{thm:allo} is given in Section \ref{sc5}.
The theorem indicates that the number of surviving banks scales differently with $\beta < 0$ (negative economy), $\beta = 0$ (neutral economy),
and $\beta > 0$ (positive economy).
Part (i) implies that the number of surviving banks is of constant order (bounded by $-2/\beta$) in a negative economy,
so equity injections as economic intervention are far from efficient.  
Part (ii) shows that in a neutral economy, the number of surviving banks scales as $\sqrt{N}$,
which is of the same order as the ``Up the River'' model \cite{Aldous02, TT18}.
Lower and upper bounds will also be derived in Section \ref{sc52}.
Part (iii) indicates that in a positive economy with no systemic interactions, 
the budget control is not efficient in the sense that 
the proportion of surviving banks is unchanged asymptotically with or without equity injections. 
The general case where $\alpha \ne 0$ seems to be challenging, 
and we conjecture that the budget's effect is also negligible.

\section{Proof of Theorem \ref{thm:existence}}
\label{sc3}

\quad The idea of the proof is similar to \cite{CRS2020}, which relies on a fixed point argument. 
Here we provide more details specific to our problem. 

Let us introduce the operator $\Gamma:\Dc\to\M$ as follows. For $\ell\in \Dc$, let $\Gamma[\ell]\in\M$ be defined by $\Gamma[\ell]_t := \P(\tau^{\ell}\le t)$ for all $t\ge 0$, where 
\b*
X^{\ell}_t := Z +\beta t + B_t- G(\ell_t) &\mbox{and}&   \tau^{\ell}:= \inf \big\{t\ge 0:~ X^{\ell}_t \le 0\big\}.
\e* 
We define similarly $\Gamma_N:\Sc(\Dc)\to \Sc(\M)$, where $\Sc(\Dc)$ (resp. $\Sc(\M)$) denotes the set of $\F-$adapted processes taking values in $\Dc$ (resp. $\M$). For $L\in\Sc(\Dc)$, define
\b*
\Gamma_N[L]_t := \frac{1}{N}\sum_{i=1}^N\mathds{1}_{\{\tau^{L}_i\le t\}},
\e*
where for $1\le i\le N$,
\b*
X^{L,i}_t := Z^i +\beta t + B^i_t- G(L_t) &\mbox{and}&  \tau^{L}_i :=  \inf \big\{t\ge 0:~ X^{L,i}_t\le 0\big\}.
\e* 
For any $\ell, \ell'\in \Dc$ (resp. $L, L'\in \Sc(\Dc)$), we write $\ell\preceq \ell'$ (resp. $L\preceq L'$) 
if $\ell_t\le \ell'_t$ for all $t\ge -1$ (resp. $L_t\le L'_t$ for all $t\ge -1$ almost surely). 
It is easy to verify that $\Gamma, \Gamma_N$ are monotone with respect to $\preceq$, i.e. $\ell\preceq \ell' \Longrightarrow \Gamma[\ell]\preceq \Gamma[\ell']$ (resp.  $L\preceq L' \Longrightarrow \Gamma_N[L]\preceq \Gamma_N[L'])$. In particular, for any $\ell\in\M$ (resp. $L\in\Sc(\M)$), $(X^{\ell},\ell)$ (resp. $(X^{L,1},\ldots, X^{L,N}, L)$) is a solution to \eqref{dyn:lim} (resp. \eqref{dyn:particle}) if and only if $\Gamma[\ell]=\ell$ (resp. $\Gamma_N[L]=L$). 

Our goal is to prove that $\Gamma$ (resp. $\Gamma_N$) has a fixed point corresponding to the minimal solution. 
To this end, we need the following lemma. 
\begin{lemma}\label{lem:usc}
For any $(\ell^n, \, n \ge 1) \subset \M$ and $\ell\in \M$ satisfying $\lim_{n\to\infty}\hr(\ell^n,\ell)=0$, we have
\b*
\limsup_{n\to\infty} \Gamma[\ell^n]_t\le \Gamma[\ell]_t, \quad \mbox{for all } t\ge 0.
\e*
\end{lemma}

\begin{proof}
Denote $g^n:=G(\ell^n)$ and $g:=G(\ell)$ for all $n\ge 1$. By Lemma \ref{lem:metrics}, there exist $\theta^n, \theta\in \Pl$ such that $\ell^n_t=\theta^n([0,t])$ and $\ell_t=\theta([0,t])$ for all $t\ge 0$. 
By the Portmanteau theorem, we get
\b* 
\limsup_{n\to\infty}\ell^n_t = \limsup_{n\to\infty}\theta^n([0,t]) \le \theta([0,t]) = \ell_t,
\e*
which yields $\Gamma[\ell]_t=\P(\tau^{\ell}\le t)\ge \P(C)$, where $C:= \{\exists s\in[0,t]: Z+\beta s+ B_s\le h_s\}$ and $h_t := G(\limsup_{n\to\infty}\ell^n_t)=\limsup_{n\to\infty}g^n_t$. Then it suffices to prove 
\b*
\limsup_{n\to\infty} \Gamma[\ell^n]_t=\limsup_{n\to\infty} \P(\tau^{\ell^n}\le t)\le \P(C).
\e* 
For each $t\ge 0$, set $h^n_t:=\sup_{k\ge n} g^k_t$ so that $g^n_t\le h^n_t\downarrow h_t$  as $n\to\infty$. Hence, 
\b*
	\limsup_{n\to\infty} \P(\tau^{\ell^n}\le t)\le 
	\lim_{n\to\infty} \P(\exists s\in[0,t]: Z+\beta s+B_s\le h^n_s)=\P(A), 
\e*
where 
\b*
A:=\{\forall n,~ \exists s\in[0,t]: Z+\beta s+ B_s\le h^n_s\} 
=\{\forall n\ge m,~ \exists s\in[0,t]: Z+\beta s+ B_s\le h^n_s\}.
\e*
Here $m$ can be any positive integer. 
Note that  $h^n$ and $h$ are non-decreasing by definition, but may not be right-continuous. It is enough to show  
 $\P(A\setminus C)=0$. Introducing an equivalent probability measure $\Q$ defined by
\b*
\left(\frac{d\Q}{d\P}\right)_t := \exp(-\beta B_t-\beta^2t/2).
\e* 
Then $\P(A\setminus C)=0\Longleftrightarrow \Q(A\setminus C)=0$,
and $(\beta s+B_s)_{0\le s\le t}$ is a Brownian motion under $\Q$. 
Thus, it suffices to prove for the case $\beta=0$. Define
\b*
\tau^n := \inf\{s\in [0,t]:~ Z+B_s \le h^n_s\},\quad \forall n\ge 1,
\e* 
 where we adopt the convention $\inf\emptyset:=\infty$. Clearly, $\tau^n\in [0,t]$ on $A$. 
By definition $h^n\downarrow h$  as $n\to\infty$, $n\mapsto\tau^n$ is non-decreasing, which implies the existence of $\tau:=\lim_{n\to\infty}\tau^n$. In particular, $\tau\in [0,t]$ also holds on $A$. 
In the following, we let the event $A$ occur and distinguish four cases.
\begin{enumerate}
\item[(i)]
Suppose $\tau^n<\tau$ for all $n\ge 1$. 
By definition,  there is some $t_n\in[\tau^n,\tau)$ such that $x+B_{t_n}\le  h^n_{t_n}\le  h^n_{\tau}$. 
Hence, $x+ B_{\tau}=\lim_{n\to\infty}(x+ B_{t_n})\le\lim_{n\to\infty}  h^n_{\tau}= h_{\tau}$. 
Thus, we get
\be\label{ineq1}
A\cap\{\tau^n<\tau,~ \forall n\ge 1\}\subseteq C. 
\ee
\item[(ii)]
If $\tau^n=\tau=t$ for some $n$, then we have $x+B_t\le h^n_t$, and further $x+B_t\le h_t$. 
Therefore, 
\be\label{ineq2}
	A\cap \{\exists n: \tau^n=\tau=t\}\subseteq C.
\ee
\item[(iii)]
If $\tau^n=\tau<t$ for some $n$, and $\tau$ is a point of discontinuity of $h$, 
then we have $x+B_{\tau}=x+B_{\tau^n} \le  h^n_{\tau^n+}= h^n_{\tau+}$. Define $\hat h_s:=\lim_{n\to\infty} h^n_{s+}$, then $x+B_{\tau}\le\hat h_{\tau}$. Namely, 
\be\label{ineq3}
x+B_d\le \hat h_d, 	
\ee
at some point $d\in D$, where $D$ is the countable subset of points of discontinuity of $h$.   
For all $0\le u<s\le t$, we have $h_s=\lim_{n\to\infty} h^n_s \ge\lim_{n\to\infty} h^n_{u+}=\hat h_u$. Now suppose that $C$ does not occur, which implies that $x+B_s> h_s$ for all $s\in[0,t]$ and hence $x+B_s> \hat h_d$ for all $d\in D$ and $s\in(d,t]$. 
By the local behavior of Brownian motion, for every  $d\in[0,t)$ the set  
$\{x+B_d\le \hat h_d,~ x+B_s>\hat h_d, \ \forall s\in(d,t]\}$ is negligible. We conclude by \eqref{ineq3} that  
\be\label{ineq4}
	\P(A\cap\{\exists n: \tau^n=\tau<t,~ \tau\in D\}\setminus C) \le \sum_{d\in D} \P(A\cap\{\exists n: \tau^n=\tau=d<t\}\setminus C) = 0. 
\ee
\item[(iv)]
Finally, suppose $\tau^{n^*}=\tau<t$ for some $n^*$, and $\tau$ is a point of continuity of $h$. 
For any $\eps>0$, there exists $\de>0$ such that $h_{\tau+\de} \le h_{\tau}+\eps$. 
Hence, for all $n\ge n^*$ large enough, $\tau^{n^*}=\tau^{n}=\tau$ and 
\b*
	x+B_{\tau}=x+B_{\tau^n} \le h^n_{\tau^n+\delta}= h^n_{\tau+\de} \to  h_{\tau+\de} \le  h_{\tau}+\eps. 
\e*
As $\eps$ is arbitrary, we get $x+B_{\tau} \le  h_{\tau}$. Therefore, 
\be\label{ineq5}
A\cap \{\exists n: \tau^n=\tau<t,~ \tau\notin D\}\subseteq C.
\ee
\end{enumerate}
Combining \eqref{ineq1}, \eqref{ineq2}, \eqref{ineq4} and \eqref{ineq5}, 
we conclude that $\P(A\setminus C)=0$ and $\limsup_{n\to\infty} \P(\tau^{\ell^n}\le t) \le \P(\tau^{\ell}\le t)$. 
\end{proof}

Next we show that $\Gamma: (\M,\hr)\to (\M,d)$ is continuous.
\begin{proposition}\label{prop:continuity}
For any $(\ell^n, \, n \ge 1) \subset \M$ and $\ell\in \M$, 
$\lim_{n\to\infty}d(\Gamma[\ell^n],\Gamma[\ell])=0$ holds when $\lim_{n\to\infty}\hr(\ell^n,\ell)=0$. 
\end{proposition}
\begin{proof}
For simplicity, we write $X^{\ell^n}\equiv X^n$, $X^{\ell}\equiv X$, $\tau^{\ell^n}\equiv \tau^n$ and $\tau^{\ell}\equiv \tau$. 
We first show that $\lim_{n\to\infty}\Gamma[\ell^n]_t=\Gamma[\ell]_t$ for all $t\in J$, 
where $J\subset \R_+$ denotes the collection of the points of continuity of $\Gamma[\ell]$. 
By Lemma \ref{lem:usc}, we have $\limsup_{n\to\infty}\Gamma[\ell^n]_t\le \Gamma[\ell]_t$ for all $t\ge 0$,
and it remains to prove $\lim_{n\to\infty}\big(\Gamma[\ell]_t-\Gamma[\ell^n]_t\big)^+ = 0$ for $t\in J$. 
Write
\b*
\big(\Gamma[\ell]_t-\Gamma[\ell^n]_t\big)^+ \le \P(\tau^{n}>t, \tau\le t) = \int_{[0,t]} \P(\tau^{n}>t|\tau=s) \d\Gamma[\ell]_s.
\e*
As in the proof of Lemma \ref{lem:usc},  we denote $g^n:=G(\ell^n)$ and $g:=G(\ell)$. 
We split the integrand in its continuous and jump part, writing $\Gamma[\ell]^c_s$ for the
continuous part, which we estimate in the following.
\b*
&&\int_{[0,t]} \P(\tau^{n}>t|\tau=s) \d\Gamma[\ell]^c_s \\
&=& \int_{[0,t]} \P(X^{n}_u>0, \forall u\le t|\tau=s) \d\Gamma[\ell]^c_s \\
&\le& \int_{[0,t]} \P(X^{n}_s +\beta(u-s)+ B_u-B_s - (g^n_u-g^n_s)>0, \forall s\le u\le t|\tau=s) \d\Gamma[\ell]^c_s \\
&\le& \int_{[0,t]} \P(X^{n}_s +\beta(u-s)+ B_u-B_s >0, \forall s\le u\le t|\tau=s) \d\Gamma[\ell]^c_s. 
\e*
On the event $\{\tau=s\}$, we have $X_s\le 0$, so
$X^{n}_s \le X^{n}_s-X_s = g_s-g^n_s$. 
We conclude, using the reflection principle (see e.g. \cite[Section 2.8A]{KS91}),
\b*
&&\int_{[0,t]} \P(\tau^{n}>t|\tau=s)\d\Gamma[\ell]^c_s \\
&\le&   \int_{[0,t]} \P(\beta(u-s)+ B_u-B_s >g^n_s-g_s, \forall s\le u\le t|\tau=s) \d\Gamma[\ell]^c_s \\
&=& \int_{[0,t]} \P\left(\inf_{s\le u\le t}\big(\beta (u-s)+ B_u-B_s\big)>g^n_s-g_s\right) \d\Gamma[\ell]^c_s \\
&=& \int_{[0,t]} \P\left(\inf_{0\le u\le t-s}\big(\beta u+ B_u\big) >g^n_s-g_s\right) \d\Gamma[\ell]^c_s \\
&=&  \int_{[0,t]} \left(\Nc\left(\frac{g_s-g^n_s+\beta (t-s)}{\sqrt{t-s}}\right)-e^{2\beta(g^n_s-g_s)}\Nc\left(\frac{g^n_s-g_s+\beta (t-s)}{\sqrt{t-s}}\right)\right) \d\Gamma[\ell]^c_s, 
\e*
where $\Nc$ denotes the cumulative distribution function of standard normal. 
Since $\lim_{n\to\infty}g^n_s=g_s$ holds in a co-countable set, the integrand above converges $\Gamma[\ell]^c-$almost everywhere to $0$. Consequently, the integral above vanishes
as $n\to\infty$ by the dominated convergence theorem.

For the integral with respect to the jump part, we get
\b*
 \int_{[0,t]} \P(\tau^{n}>t|\tau=s) \d\big(\Gamma[\ell]_s-\Gamma[\ell]^c_s\big) &=&  \sum_{s\le t} \P(\tau^{n}>t|\tau=s) \Delta \Gamma[\ell]_s \\
 &=&  \sum_{s<t} \P(\tau^{n}>t|\tau=s)\P(\tau=s) \\
&=&  \sum_{s<t} \P(\tau^{n}>t, \tau=s), 
\e*
where we may take the sum over $s<t$ because $t\in J$. For any $s>0$, we have by the Portmanteau theorem
\b*
\liminf_{n\to\infty} \ell^n_{s-} \ge \ell_{s-} &\Longrightarrow& \liminf_{n\to\infty} g^n_{s-} \ge g_{s-}.
\e*
For $s<t$, we  have $\P(\tau^{n}>t, \tau=s) \le \P(\forall \eps \in (0,t-s): Z+\beta (s+\eps)+ B_{s+\eps}\ge  g^n_{(s+\eps)-}, \tau=s)$. 
Taking the $\limsup$ yields, by  Fatou's lemma,
\b*
\limsup_{n\to\infty} \P(\tau^{n}>t, \tau=s) &\le& \P(\forall \eps \in (0,t-s): Z+\beta (s+\eps)+ B_{s+\eps} \ge g_{(s+\eps)-}, \tau=s) \\
&\le &\P(Z+\beta s+ B_s\ge g_s, \tau=s) \\
&\le &\P(Z+\beta s+ B_s= g_s) = 0.
\e*
Again by the dominated convergence theorem,
we get that the sum converges to zero as $n\to\infty$, which proves $\lim_{n\to\infty}\Gamma[\ell^n]_t=\Gamma[\ell]_t$. Hence, $\lim_{n\to\infty}\hr(\Gamma[\ell^n],\Gamma[\ell])=0$. 

To show $\lim_{n\to\infty}d(\Gamma[\ell^n],\Gamma[\ell])=0$, it remains to show $\lim_{n\to\infty}\Gamma[\ell^n]_{\infty}=\Gamma[\ell]_{\infty}$. 
This is straightforward when $\beta\le 0$, and we assume $\beta>0$ in the remaining of proof. 
For each $z\in\R$, denote by $p(z)$ the probability that $z+\beta t+B_t$ never hits $(-\infty, 0]$. 
We have
$p(z)=1-e^{-2\beta z^+}$, and further $\lim_{z\to\infty}p(z)=1$.
For each $t\ge 0$, we get
\b*
\big|\Gamma[\ell^n]_{\infty}-\Gamma[\ell]_{\infty}\big| &=& \big|\mathbb P(\tau^n=\infty)-\mathbb P(\tau=\infty)\big| \\
&\le& \big|\mathbb P(\tau^n>t)-\mathbb P(\tau^n=\infty)\big|+\big|\mathbb P(\tau^n>t)-\mathbb P(\tau>t)\big|+\big|\mathbb P(\tau>t)-\mathbb P(\tau=\infty)\big| \\
&=& \big|\mathbb P(\tau^n>t)-\mathbb P(\tau^n=\infty)\big|+\big|\Gamma[\ell^n]_t-\Gamma[\ell]_t\big|+\big|\mathbb P(\tau>t)-\mathbb P(\tau=\infty)\big|.
\e*
Let us estimate the first and third terms. 
\b*
\big|\mathbb P(\tau^n>t)-\mathbb P(\tau^n=\infty) \big| &=& \E[\mathds 1_{\{\tau^n>t\}}]- \E[\mathds 1_{\{\tau^n>t\}}\mathds 1_{\{\tau^n=\infty\}}] \\
&=& \E[\mathds 1_{\{\tau^n>t\}}]- \E\big[\P(\tau^n=\infty|\tau^n>t)\mathds 1_{\{\tau^n>t\}}\big].
\e*
On the event $\{\tau^n>t\}$, we have
\b*
\{\tau^n=\infty\} &=& \{X^n_t+\beta(s-t)+(B_s-B_t)- (g^n_s-g^n_t)>0,\forall s\ge t \} \\
&\supset &  \{X^n_t-c+\beta(s-t)+(B_s-B_t)>0,\forall s\ge t \}, \quad \mbox{with } c:=G(1),
\e*
which yields by the Markov property $\P(\tau^n=\infty|\tau^n>t)\ge p(X^n_t-c)$ and thus
\b*
\big|\mathbb P(\tau^n>t)-\mathbb P(\tau^n=\infty) \big| = \E\big[\mathds 1_{\{\tau^n>t\}}\big(1-\P(\tau^n=\infty|\tau^n>t)\big)\big] 
\le  \E\big[\mathds 1_{\{\tau^n>t\}}\big(1-p(X^n_t-c)\big)\big].
\e*
Now we couple the process $X_t$ by $Y_t:=Z+\beta t+B_t$ for all $t\ge 0$, 
so $X^n_t\ge Y_t-c$.
Let $\sigma:=\inf\{t\ge 0: Y_t\le 0\}$, and we have $\tau^n\le \sigma$.
Hence, 
\b*
\big|\mathbb P(\tau^n>t)-\mathbb P(\tau^n=\infty)  \big|\le \E\big[\mathds 1_{\{\tau^n>t\}}\big(1-p(X^n_t-c)\big)\big] \le \big| \E\big[\mathds 1_{\{\sigma>t\}}\big(1-p(Y_t-2c)\big)\big].
\e*
Similarly, we can establish 
$\big|\mathbb P(\tau>t)-\mathbb P(\tau=\infty)\big| \le \big| \E\big[\mathds 1_{\{\sigma>t\}}\big(1-p(Y_t-2c)\big)\big]$, and thus
\b*
 \big|\mathbb P(\tau^n=\infty)-\mathbb P(\tau=\infty)\big|  \le  2\E\big[\mathds 1_{\{\sigma>t\}}\big(1-p(Y_t-2c)\big)\big] + \big|\Gamma[\ell^n]_t-\Gamma[\ell]_t\big|.
 \e*
As $\lim_{t\to\infty}Y_t=\infty$ holds almost surely, we get by the dominated convergence theorem,
\b*
\lim_{t\to\infty}\E\big[\mathds 1_{\{\sigma>t\}}\big(1-p(Y_t-2c)\big)\big] = 0.
\e* 
For any $\eps>0$, there exists $t_{\eps}$ such that 
$\E\big[\mathds 1_{\{\sigma>t\}}\big(1-p(Y_t-2c)\big)\big] \le \eps$ for all $t\ge t_{\eps}$. Fix an arbitrary $t\in J$ with $t>t_{\eps}$. 
Then we have 
\b*
\lim_{n\to\infty}  \big|\Gamma[\ell^n]_{\infty}-\Gamma[\ell]_{\infty}\big| \le  2\eps + \lim_{n\to\infty} \big|\Gamma[\ell^n]_{t}-\Gamma[\ell]_{t}\big|
=2\eps,
\e*
which yields the desired result. 
\end{proof}

Now we prove Theorem \ref{thm:existence}.

\begin{proof}[Proof of Theorem \ref{thm:existence}]
\rmi We start with the particle system \eqref{dyn:particle}. 
For each $n\ge 0$, denote by  $\Gamma_N^{(n)}$ the $n$-th iterate of the operator $\Gamma_N$. First we show $\dL^N:=\Gamma_N^{(N)}[{0}]$ is the minimal solution to \eqref{dyn:particle}, where $0\in\M$ is the function that is identically equal to zero. The monotonicity of $\Gamma_N$, combined with the fact $0\preceq \Gamma_N[0]=\Gamma_N^{(1)}[0]$, yields $\Gamma_N^{(1)}[0]\preceq \Gamma_N\big[\Gamma_N^{(1)}[0]\big]=\Gamma_N^{(2)}[0]$, 
and thus by induction $\Gamma_N^{(n-1)}[0]\preceq \Gamma_N^{(n)}[0]$ for all $n\ge 1$. Note that $\dLa^N\in\M$ only takes values in $\{k/N: k=0, \ldots, N\}$. For $0\le n\le N$, define the stopping times
\b*
\sigma_n := \inf\left\{t\ge 0:~ \Gamma_N^{(n)}[{0}]_t \ge \frac{n}{N}\right\}\in [0,\infty].
\e*
We claim that  \eqref{eq:agree} holds for all $1\le n\le N$:
\be\label{eq:agree}
\Gamma_N^{(n-1)}[{0}]_t = \Gamma_N^{(n)}[{0}]_t,\quad \mbox{for all }  t<\sigma_n.
\ee
Clearly, $\Gamma_N^{(0)}$ denotes the identity operator. For $n=1$, observe that  the first jump time $\sigma_1$ of $\Gamma_N^{(0)}[0]$ coincides with the first time at which any of the drifted Brownian motions $Z^i+\beta t+B^i_t$ hits zero, and $\Gamma_N^{(1)}[0]$ is equal to zero prior to the time $\sigma_1$,  regardless of $\sigma_1<\infty$ and $\sigma_1=\infty$. 
This means that $\Gamma_N^{(0)}[0]=0=\Gamma_N^{(1)}[0]$ holds for all  $t<\sigma_1$. 
For the inductive step, assume that  \eqref{eq:agree} is true for all positive integers up to $n$. 
Applying $\Gamma_N$ to both sides, we obtain
\b*
\Gamma_N^{(n)}[{0}]_t =  \Gamma_N^{(n+1)}[{0}]_t,\quad \mbox{for all }  t<\sigma_n.
\e*
We are done if $\sigma_n=\infty$, and it suffices to deal with the case $\sigma_n<\infty$. 
We distinguish two cases. In the first case, we suppose $\Gamma_N^{(n)}[{0}]_{\sigma_n}>n/N$. 
Therefore, we have $\Gamma_N^{(n)}[{0}]_{\sigma_n}\ge (n+1)/N$, and hence $\sigma_n=\sigma_{n+1}$, completing the inductive step.
For the second case, we have $\Gamma_N^{(n)}[{0}]_{\sigma_n}=n/N$, and for $t\in (\sigma_n,\sigma_{n+1})$,
\b*
\frac{n}{N}= \Gamma_N^{(n)}[{0}]_{\sigma_n} \le \Gamma_N^{(n)}[{0}]_{t}\le\Gamma_N^{(n+1)}[{0}]_{t} < \frac{n+1}{N},
\e*
which implies that $\Gamma_N^{(n)}[{0}]$ and $\Gamma_N^{(n+1)}[{0}]$ agree on  $(0,\sigma_{n+1})$, completing the inductive
step. 
Next we prove that $\dL^N=\Gamma_N^{(N)}[{0}]$ solves the particle system. For $n\ge 0$, repeatedly applying $\Gamma_N$, we get
\b*
\Gamma_N^{(n)}[{0}]_t = \Gamma_N^{(n+1)}[{0}]_t= \cdots = \dL^N_t= \Gamma_N[\dL_N]_t,\quad \mbox{for all } t<\sigma_n.
\e*
Choosing $n=N$, we obtain $\dL^N_t= \Gamma_N[\dL_N]_t$ on $[0,\sigma_N)$. 
If $\sigma_N=\infty$, then $\dL^N=\Gamma_N[\dL_N]$. 
Otherwise, recall that $\dL^N\preceq \Gamma_N[\dL^N]$. 
So we see that $1\le \dL^N_{\sigma_N}= \Gamma_N[\dL^N]_{\sigma_N}\le 1$, and thus $\dL^N= \Gamma_N[\dL^N]$. 
To prove the minimality of $\dL^N$, we pick an arbitrary solution $L$. 
By definition, we have $0\preceq L$, and thus by iteration, 
\b*
\dL^N=\Gamma_N^{(N)}[0]\preceq \Gamma_N^{(N)}[L]=L,
\e*
which implies that $\dL^N$ is indeed the minimal solution.  

\medskip

\noindent \rmii Next we turn to the McKean-Vlasov equation \eqref{dyn:lim}. 
Define the sequence $(\Gamma^{(n)}[{0}], \, n \ge 1) \subset \M$, 
where $\Gamma^{(n)}$ denotes the $n$-th iterate of $\Gamma$. 
By construction, we see that the sequence is non-decreasing with respect to $\preceq$. 
We show in the following that $\Gamma^{(n)}[{0}] \stackrel{\hat d}{\to} \dLa\in \M$ exists, and $\dLa$ is the minimal solution to \eqref{dyn:lim}. 
By definition, we have ${0}\preceq \Gamma[{0}]$. 
For each $t\ge 0$, the sequence $\big(\Gamma^{(n)}[{0}]_t, \, n \ge 1\big)$ is non-decreasing and lying in $[0,1]$, which
implies that we can define $\tilde{\Lambda}$  to be its pointwise limit
\b*
\tilde{\Lambda}_t : = \lim_{n\to\infty} \Gamma^{(n)}[{0}]_t,\quad \mbox{for all } t\ge 0.
\e*
Clearly, $\tilde{\Lambda}$ is non-decreasing with 
$0\le \tilde{\Lambda}_t\le 1$, so its modification $\dLa_t:=\tilde{\Lambda}_{t+}$ lies in $\M$. 
This implies that $\Gamma^{(n)}[{0}] \stackrel{\hr}{\to}\dLa$, and further by Proposition \ref{prop:continuity},
\b*
\Gamma[\dLa] = \Gamma\left[\lim_{n\to\infty}\Gamma^{(n)}[{0}]\right] =\lim_{n\to\infty}\Gamma^{(n+1)}[{0}] = \dLa.
\e* 
So $\dLa$ solves \eqref{dyn:lim}. Suppose $\Lambda$ is another solution. 
By definition, it holds
that ${0}\preceq \Lambda$, and using the monotonicity of $\Gamma$ this leads to
$\Gamma[{0}]\preceq \Gamma[\Lambda]=\Lambda$. 
A straightforward induction shows that $\Gamma^{(n)}[{0}]\preceq \Lambda$ for all $n\ge 1$. 
If $t$ is a point of continuity of $\dLa$, 
then $\dLa_t=\lim_{n\to\infty} \Gamma_t^{(n)}[{ 0}] \le \Lambda_t$, and further $\dLa\preceq \Lambda$ by the right continuity of $\dLa$. This proves that $\dLa$ is the minimal solution. 

\medskip

\noindent \rmiii For any solution $(X,\Lambda)$ of \eqref{dyn:lim} that is different to $(\dX,\dLa)$,  there exist $v>u>0$ and $b>a>0$ such that $\Lambda_t >b>a>\dLa_t$ for all $t\in [u,v]$. Recall that $\underline{\tau}:=\inf\{t\ge 0: \dX_t\le 0\}$ 
and $\tau:=\inf\{t\ge 0: X_t\le 0\}$, and we want to show 
\b*
\dLa_{\infty}<\Lambda_{\infty} &\Longleftrightarrow& \P(\tau=\infty)<\P(\underline\tau=\infty),
\e*
which is implied by the fact that $\P(\tau<\infty=\underline\tau)>0$. 
Denote by $Y_t:=Z+\beta t+ B_t$ for $t\ge 0$, then
\b*
\{\tau<\infty=\underline \tau\} \supseteq \left\{Y_t>\dLa_t: \forall t\in [0,u]\right\} \cap \left\{Y_t\in [a,b]: \forall t\in [u,v]\right\}  \cap\left \{Y_t>\dLa_t: \forall t\ge v\right\}, 
\e*
which yields the desired result by the Markov property. 
\end{proof}

\section{Proof of Theorems \ref{thm:limit} and \ref{thm:reg}}
\label{sc4}

\subsection{Proof of Theorem \ref{thm:limit}}
\label{sc41}

\quad Recall that $\Cc$ (resp. $\M$) is endowed with the metric $d_u$ (resp. $\hr$).
Let $\Xi:= \Cc \times \M$ be endowed with the product topology, and we denote the corresponding  metric  by $d_u\otimes \hr$. Then $\Xi$ is Polish. Define accordingly the map $\iota:  \Xi\to \Dc$ by $\iota(f,\ell):=f-G(\ell)$. 
For every $t\ge -1$, denote by $\lambda_t$ the path functionals  on $\Dc$:
\b*
\lambda_t(x) := \mathds 1_{\{\tau(x)\le t\}} &\mbox{with}&
\tau(x):=\inf\{s\ge 0: x_s\le 0\}.   
\e*
Set $\lambda(x):=(\lambda_t(x), \, t \ge -1) \in \M$. We say that $x$ satisfies the {\em crossing property} if
  \b*
  \inf_{0\le s\le h} x_{\tau+s}-x_{\tau} < 0,\quad \mbox{for all }  h>0,
  \e*
where we abbreviate $\tau(x)$ by $\tau$. 
The lemma below summarizes the properties of $\iota$ and $\lambda$, which is an adaptation from \cite{NS2019}. 
We provide its proof for completeness. 
\begin{lemma}
\label{lem:conti}
\begin{enumerate}
\item[(i)]
The embedding $\iota:(\Xi,d_u\otimes\hr)\to(\Dc,d_m)$ is continuous. 
\item[(ii)]
Let $x\in \iota(\Xi)$ satisfy the crossing property. Then for any sequence $(x^n, \, n \ge 1) \subset \iota(\Xi)$ converging to $x$ under $d_m$, 
we have $\lim_{n\to\infty}\lambda_t(x^n)=\lambda_t(x)$ for all $t\ge 0$ in a co-countable set.  
\item[(iii)]
Assume that $(\eta^n)_{n\ge 1}\subset\Pc(\Xi)$ is a convergent sequence with limit $\eta$. Define $\nu^n:=\iota(\eta^n)$
 and $\nu:=\iota(\eta)$. If $\nu-$almost every path satisfies the crossing property, i.e. 
 \b*
 \nu\left(\left\{x\in\Dc:~   \inf_{0\le s\le h} x_{\tau+s}-x_{\tau}<0,~ \forall h>0\right\}\right) = 1,
 \e*
then 
\b*
\left(\int_{\Dc}\lambda_t(x)\d \nu^n(x), \, t \ge -1 \right)=:\langle \nu^n, \lambda \rangle\stackrel{\hr}{\to} \langle \nu,
 \lambda \rangle:=\left(\int_{\Dc}\lambda_t(x)\d \nu(x), \, t \ge -1\right).
 \e* 
 \end{enumerate}
\end{lemma}
\begin{proof}
\rmi Let $(f^n,\ell^n)\to (f,\ell)$ in $\Xi$ and $T$ be any point of continuity of $\ell$. 
Then we have $f^n\to f$ in $\Cc([-1,T])$, and thus in $\Dc([-1,T])$. 
Moreover, as $G(\ell^n)$ and $G(\ell)$ are non-decreasing and $\lim_{n\to\infty}G(\ell^n_t)=G(\ell_t)$ for each point of continuity $t$ of $G(\ell)$,  \cite[Corollary 12.5.1]{Whitt2002} shows that $G(\ell^n)\to G(\ell)$ in $\Dc([-1,T])$. 
The conclusion follows from \cite[Theorem 12.7.3]{Whitt2002}.

\medskip

\noindent \rmii Let $J$ be a co-countable set consisting of all points of continuity of $x$. 
Recall that \cite[Theorem 13.4.1]{Whitt2002} ensures $\lim_{n\to\infty} \inf_{-1\le s\le t}x^n_s =  \inf_{-1\le s\le t}x_s$ 
for all $t\in J$.
Hence, for any $t\in J$ such that $\inf_{-1\le s\le t}x_s\neq 0$ we have
\b*
\lim_{n\to\infty}\lambda_t(x^n) = \lim_{n\to\infty} {\mathds 1}_{\{\inf_{-1\le s\le t}x^n_s\le 0\}} =  {\mathds 1}_{\{\inf_{-1\le s\le t}x_s\le 0\}}=\lambda_t(x). 
\e*
Now let  $t\in J$ with $\inf_{-1\le s\le t}x_s=0$. Then $t\ge \tau(x)$ by definition. If $t>\tau(x)$, then the crossing property implies that $\inf_{-1\le s\le t}x_s<0$, which leads to a contraction. 
It follows that the only point in $J$ where the convergence may fail is $\tau(x)$, so $J\setminus \{\tau(x)\}$ is still a co-countable set on which we have the desired convergence.

\medskip

\noindent \rmiii For ease of presentation, denote $\langle \nu^n,\lambda_t \rangle=:\ell^n_t$ and $\ell^n:=(\ell^n_t)_{t\ge 0}\in\M$. 
By the compactness of $\M$,
we may assume that $\ell^n_t\to\ell_t$ for some $\ell\in\M$ at every point of continuity $t$ of $\ell$.
Then it suffices to prove $\ell=\langle \nu,\lambda\rangle$. Let $T>0$ and $g:[0,T]\to\R$ be bounded and measurable. 
The dominated convergence theorem implies 
 \b*
 \lim_{n\to\infty}\int_0^T\ell^n_tg_t\d t = \int_0^T\ell_tg_t\d t. 
 \e*
 By the Skorokhod representation theorem (see e.g. \cite[Theorem 4.30]{Kallenberg02}), we can write 
 \b*
\int_0^T\ell^n_tg_t\d t =\E\left[ \int_0^T\lambda_t(Y^n)g_t\d t\right], 
 \e*
 where $Y^n$ converges almost surely in $\Dc$ to $Y$ with $\Lc(Y^n)= \nu^n$ and $\Lc(Y)=\nu$. 
 By {\rm (ii)}, it holds with probability one that
 \b*
 \lim_{n\to\infty} \int_0^T\lambda_t(Y^n)g_t\d t = \int_0^T\lambda_t(Y)g_t\d t. 
 \e*
Using the dominated convergence theorem again, we get
 \b*
 \int_0^T\ell_t g_t \d t= \lim_{n\to\infty} \E\left[\int_0^T\lambda_t(Y^n)g_t\d t\right] = \E\left[\int_0^T\lambda_t(Y)g_t \d t\right]= \int_0^T\langle \nu,\lambda_t\rangle g_t\d t.
 \e*
It follows from the right-continuity of $\ell$ that 
 $\ell=\langle \nu,\lambda\rangle$ for all $t\le T$. 
Since $T>0$ is arbitrary, the claim follows.
\end{proof}

Let $(X^{N,1},\ldots, X^{N,N}, L^N)$ be a solution to \eqref{dyn:particle}. 
We define the empirical measures $(\xi^N, \, N \ge 1)$ and $(\mu^N, \, N \ge 1)$ taking values in $\Pc(\Xi)$ and $\Pc(\Dc)$ by
 \be\label{def:empirical}
  \xi^N:=\frac{1}{N}\sum_{i=1}^N\delta_{(F^{N,i}, L^N)} &\mbox{and}&  \mu^N:=\frac{1}{N}\sum_{i=1}^N\delta_{X^{N,i}},
 \ee  
 where $F^{N,i}:=(Z^{N,i}+\beta t+ B^i_t, \, t \ge 0)$. 
By definition, we have $\mu^N=\iota(\xi^N)$ for all $N\ge 1$. 
We need a few preliminary results on the sequence of random measures $(\xi^N, \, N \ge 1)$. 
\begin{lemma}\label{prop:tight}
The sequence $(\xi^N)_{N\ge 1}$ is tight under Assumption \ref{ass1}.
\end{lemma} 
\begin{proof}
The assumption yields the tightness of $(Z^{N,1}+\beta t+B^1_t, \, t \ge 0)$, $N\ge 1$, and thus that of 
$((Z^{N,1}+\beta t+B^1_t, \, t \ge 0), \, L^N)$, $N \ge 1$ as $(\M,\hat d)$ is compact. 
By Sznitman's theorem (see e.g. \cite[Proposition 2.2]{Sznit91}),  $(\xi^N, \, N \ge 1)$ is tight. 
\end{proof}
 \begin{proposition}\label{lem:lim}
 Let $\xi$ be any limit point of $(\xi^N)_{N\ge 1}$ and $\mu:=\iota(\xi)$. 
\begin{enumerate}
\item[(i)]
For almost every realization $\omega$, if  $\Lc((F,L))=\xi(\omega)$, then $(F_t-F_0-\beta t, \, t \ge 0)$ is a Brownian motion with respect to the filtration generated by $(F,L)$ In particular,  $(F_t-F_0-\beta t, \, t \ge 0)$ is independent of $F_0$.
 \item[(ii)]
 $\mu$ satisfies the crossing property with probability one.
\end{enumerate}
\end{proposition}

 \begin{proof}
\rmi For any $t>s\ge 0$, $n\in\N$, $0\le s_1\le \cdots\le s_n\le s$ and continuous and bounded functions $\phi_k: \R^2\to\R$, denote $\Phi(f,\ell):=\Pi_{k=1}^n \phi_k(f_{s_k}, \ell_{s_k})$ for all $(f,\ell)\in \Xi$. 
We have
 \b*
&& \E\left[\left(\int_{\Xi}(f_t-f_s-\beta (t-s))G(f,\ell) \d\xi^N(f,\ell)\right)^2\right] \\
&=&  \E\left[\left(\frac{1}{N}\sum_{i=1}^N (B^i_t-B^i_s)\Phi(W^{N,i}, L^N)\right)^2\right] \\
&=&  \frac{1}{N} \E\left[\left((B^1_t-B^1_s)\Phi(W^{N,1}, L^N)\right)^2\right]+  \frac{N^2-N}{N^2}\E\left[(B^1_t-B^1_s)(B^2_t-B^2_s)\Phi(W^{N,1}, L^N)\Phi(W^{N,2}, L^N)\right] \\
&=&  \frac{1}{N} \E\left[\left((B^1_t-B^1_s)\Phi(W^{N,1}, L^N)\right)^2\right] \le \frac{\|\Phi\|_{\infty}(t-s)}{N} \longrightarrow 0,
 \e*
 where $W^{N,i}=(W^{N,i}_t, \,  t \ge 0)$ with $W^{N,i}_t:=Z^{N,i}+\beta t+ B^i_t$. 
Thus, for almost every realization $\omega$, if $(F,L)\sim \xi(\omega)$, then $(F_t-F_0-\beta t, \, t \ge 0)$ is a martingale with respect to the filtration generated by $(F,L)$. 
Furthermore, repeating the same arguments to the process $((F_t-F_0-\beta t)^2-t, \,  t \ge 0)$, we obtain that 
  \b*
 \E\left[\left(\int_{\Xi}[(f_t-f_0-\beta t)-(f_s-f_0-\beta s)-(t-s)]\Phi(f,\ell) \d\xi(f,\ell)\right)^2\right] = 0,
\e*
which implies that $((W_t-W_0-\beta t)^2- t, \, t \ge 0)$ is also a martingale with respect to the filtration generated by $(W,L)$. We conclude by  L\'evy's characterization (see e.g. \cite[Theorem 18.3]{Kallenberg02}).
 
 \vspace{1mm}
 
\noindent  \rmii For simplicity, we assume that $\Lc(\mu^N)$ converges to $\Lc(\mu)$. Note that $\Lc(\mu)=\Lc(\iota(\xi))$. 
For any fixed $h>0$, we have
  \b*
&&  \E\left[ \mu\left(\left\{x\in\Dc: \inf_{0\le s\le h} x_{\tau+s}-x_{\tau} =0\right\} \right)\right] \\
  &=& \E\left[ \xi\left(\left\{(f,\ell)\in \Xi: \inf_{0\le s\le h} f_{\tau+s}-f_{\tau} - (G(\ell_{\tau+s})-G(\ell_{\tau})) =0\right\} \right)\right] \\
  &\le &  \E\left[ \xi\left(\left\{(f,\ell)\in \Xi: \inf_{0\le s\le h} f_{\tau+s}-f_{\tau} =0\right\} \right)\right]. 
  \e*
As for almost every realization $\omega$, if $\Lc((F,L))=\xi(\omega)$, then $\tau=\tau(F-G(L))$ is a stopping time   with
respect to the filtration generated by $(F,L)$. Since $F-F_0$ is a Brownian motion with
respect to the same filtration, the strong Markov property yields
\b*
 \E\left[ \xi\left(\left\{(f,\ell)\in \Xi: \inf_{0\le s\le h} f_{\tau_0+s}-f_{\tau_0} =0\right\} \right) \right]=\P\left(\inf_{0\le s\le h}F_s=0\right)=0. 
\e*
Hence, 
  \b*
\mu\left(\left\{x\in\Dc: \inf_{0\le s\le h} x_{\tau_0+s}-x_{\tau_0} =0\right\} \right)=0,
\e*
holds almost surely. Repeating this reasoning for $h =h_n$ for a sequence $(h_n, \, n \ge 1) \subset \R_+$ such that $\lim_{n\to\infty}h_n=0$ yields the result.
\end{proof}

Following the arguments in \cite{DIRT2015b},
we have the following proposition which establishes the convergence of random measures. 
(see also \cite{CRS2020, NS2019}).

\begin{proposition}
The sequence $(\xi^N, \, N \ge 1)$ is tight under Assumption \ref{ass1}. 
Moreover, let $\xi$ be any limit point. Then for almost every realization $\omega$, $\xi(\omega)$ coincides with the law of $(F,\Lambda)$, where $(F-G(\Lambda), \Lambda)$ is a solution to \eqref{dyn:lim}.
\end{proposition}
\begin{proof}
The tightness of $(\xi^N, \, N \ge 1)$ follows from  Lemma \ref{prop:tight}. 
Without loss of generality, we may assume the convergence of
$(\xi^N, \, N \ge 1)$ for ease of presentation. 
Then we get $\mu^N\to \mu:=\iota(\xi)$ by the continuity of $\iota$. 
Consider the maps 
 \b*
 t\mapsto \E[\langle \mu, \lambda_t\rangle] &\mbox{and}& t\mapsto \E\left[\int_{\Xi} \ell_t \d\xi(f,\ell)\right],
  \e*
 that are non-decreasing and thus have at most countably discontinuities. 
 Denote by $J$ the co-countable set of their points of continuity.   
 Combining Lemma \ref{lem:conti}  {\rm (iii)} and Proposition \ref{lem:lim} {\rm (iii)}, 
 we deduce that $ \langle \mu^N, \lambda\rangle \to \langle \mu, \lambda\rangle$ in $\M$ almost surely. 
 For any $t\in J$, we have $\E[\langle \mu, \lambda_t\rangle]=\E[\langle \mu, \lambda_{t-}\rangle]$,
 and thus $\langle \mu, \lambda_t\rangle=\langle \mu, \lambda_{t-}\rangle$ almost surely. 
 Therefore, $ \langle \mu^N, \lambda\rangle \to \langle \mu, \lambda\rangle$ in $\M$ which yields 
  \be\label{eq:cont}
  \lim_{N\to\infty} \langle \mu^N, \lambda_t\rangle=\langle \mu, \lambda_t\rangle.
  \ee
Let $\pi_0:\Xi\to\R$ be the projection defined by $\pi(f,\ell)=f_0$. By definition, $\pi_0$ is continuous, and hence
\b*
\lim_{N\to\infty} \frac{1}{N}\sum_{i=1}^N\delta_{Z^{N,i}} = \lim_{N\to\infty} \pi_0(\xi^N) = \pi_0(\xi), 
\e* 
which implies that for almost every realization $\omega$, $\Lc(\pi_0(\xi(\omega)))=\lim_{N\to\infty}\theta^N=\theta$. 
By Lemma \ref{lem:lim} {\rm (i)}, for almost every realization $\omega$, if $\Lc(F,L)=\xi(\omega)$, then $(F_t-\beta t, \, t \ge 0)$ is Brownian motion with $\Lc(F_0)=\theta$. Hence, it remains to prove
\b*
\Lambda_t = \mathbb P\big(\tau(F-G(L))\le t\big) = \langle \mu(\omega),\lambda_t\rangle,\quad \mbox{for all } t\ge 0.
\e*
To this end, note for every $t\in J$,
  \b*
  \E\left[ \int_{\Xi}\big| |\ell_t-\langle \mu, \lambda_t \rangle| - |\ell_t-\langle \mu^N, \lambda_t \rangle| \big| \d\xi^N(f,\ell)\right] \le \E\left[|\langle \mu, \lambda_t \rangle -\langle \mu^N, \lambda_t \rangle|\right]
  \e* 
  which vanishes as $N\to\infty$ by \eqref{eq:cont} and the dominated convergence theorem. 
  As $t\in J$, we have $\ell_{t-}=\ell_t$ for  $\xi-$almost every $\ell$, and thus the map $\M\ni\ell\to\ell_t\in\R_+$ is continuous $\xi-$almost surely. It follows
that
  \b*
  \E\left[ \int_{\Xi} |\ell_t-\langle \mu, \lambda_t \rangle| \d\xi(f,\ell)\right] &=& \lim_{N\to\infty}  \E\left[ \int_{\Xi} |\ell_t-\langle \mu, \lambda_t \rangle| \d\xi^N(f,\ell)\right] \\
 & \le& \lim_{N\to\infty}  \E\left[ \int_{\Xi} |\ell_t-\langle \mu^N, \lambda_t \rangle| \d\xi^N(f,\ell)\right] ~~=~~ 0, 
  \e*  
 where the last equality holds since $\langle \mu^N,\lambda\rangle = \ell$ holds for $\xi^N-$almost every $\ell\in\M$, almost surely. 
The conclusion follows by letting $t$ range over a countable dense subset of $J$ and using the right-continuity of $\ell$.
\end{proof}

Now we  prove Theorem \ref{thm:limit}.
\begin{proof}[Proof of Theorem \ref{thm:limit}]
First, we show that any limit point $\xi$ is equal to $\Lc(\dX,\dLa)$ almost surely. Consider the perturbed system: for $i=1,\ldots, N$,
\b*
\tilde X_t^{L, i}=\tilde Z^{N,i}+\beta t +B^i_t - G(L_t),\quad \mbox{for all } t\ge 0
\e*
and
\b*
\tilde \Gamma_N[L]_t:=\frac{1}{N}\sum_{i=1}^N {\mathds 1}_{\{\tilde\tau^L_i\le t\}} &\mbox{and}& \tilde\tau^L_i:=\inf\{t\ge 0: \tilde X^{L,i}_t \le 0\},
\e*
where $\Lc(\tilde Z^{N,i})=\theta$ for $1\le i\le N$. Taking $L\equiv \dLa$, we have 
\b*
\tilde X^{\dLa, i}_t=\tilde Z^{N,i}+\beta t +B^i_t - G(\dLa_t),\quad \mbox{for all } t\ge 0,
\e*
and $\tilde\tau^{\dLa}_i$ are i.i.d. random variables that are identical in law to $\underline\tau:=\inf\{t\ge 0: \dX_t\le 0\}$. 
It follows from the inequality of Dvoretzky-Kiefer-Wolfowitz \cite{Massart90} that,
\b*
\mathbb P\left(\|\tilde\Gamma_N[\dLa]-\dLa\|_{\infty}>\gamma_N\right) \le 2e^{-2N\gamma_N^2},\quad \mbox{for all } N\ge 1.
\e*
For each $N\ge 1$, define $A_N:=\{\|\tilde\Gamma_N[\dLa]-\dLa\|_{\infty}\le \gamma_N\}$ and the operator $\hat{\Gamma}_N$ by 
\b*
\hat\Gamma_N[L]_t := \frac{1}{N}\sum_{i=1}^N\mathds{1}_{\{\hat\tau^{L}_i\le t\}},
\e*
where for $i=1,\ldots, N$,
\b*
\hat X^{L,i}_t = \tilde Z^{N,i}+\alpha(\gamma_N)+\beta t + B^i_t- G(L_t) &\mbox{and}&
\hat\tau^{L}_i =  \inf \big\{t\ge 0:~ \hat X^{L,i}_t\le 0\big\}. 
\e* 
In particular, $\Lc(\tilde Z^{N,i}+\alpha(\gamma_N))=\theta^N=\Lc(Z^{N,i})$ for $1\le i\le N$.
On the set $A_N$, we have
\b*
\dLa_t\ge \tilde\Gamma_N[\dLa]_t-\gamma_N,\quad \mbox{for all } t\ge 0.
\e*
Hence, $\hat\Gamma_N[0]-\gamma_N \preceq \tilde\Gamma_N[-\gamma_N]-\gamma_N\preceq \dLa$ holds as 
$-G(-\gamma_N)-\alpha(\gamma_N) \le -G(0) = 0$.
Using the monotonicity again, we get
\b*
 \hat\Gamma_N^{(2)}[0]-\gamma_N= \hat\Gamma_N[ \hat\Gamma_N[0]]-\gamma_N\preceq \tilde\Gamma_N[ \hat\Gamma_N[0]-\gamma_N]-\gamma_N \preceq  \tilde\Gamma_N[\dLa]-\gamma_N \preceq \dLa,
\e*
where the second inequality still follows from the uniform continuity of $G$, i.e. 
\b*
-G\big(\hat\Gamma_N[0]_t-\gamma_N\big) \le \alpha(\gamma_N) -G\big(\hat\Gamma_N[0]_t\big),\quad \mbox{for all } t\ge0.
\e*
Repeating this arguments, 
we obtain $\hat\Gamma_N^{(n)}[0]-\gamma_N\preceq \dLa$ for all $n\ge 1$, and finally $\dL^N - \gamma_N\preceq \dLa$. 
Hence, it holds with probability one that
\be\label{eq:limsup}
\limsup_{N\to\infty}\dL^N \preceq \dLa.
\ee
Take an arbitrary limit point $\xi$ of a convergent subsequence, still denoted by $(\xi^N, \, N \ge 1)_{N\ge 1}$ for simplicity. 
It follows that for almost every realization $\omega$, 
\b*
\frac{1}{N}\sum_{i=1}^N \delta_{\big(\dX^{N,i}(\omega),\dL^N(\omega)\big)} \to \xi(\omega) &\mbox{in}& \Pc(\Xi).
\e*
On the other hand, there exists a solution $(F,\Lambda)$ to \eqref{dyn:lim} such that $\xi(\omega)=\Lc(F,\Lambda)$. 
Therefore, $\dL^N(\omega) \stackrel{\hr}{\to} \Lambda$ and 
\b*
\dLa_t\ge \lim_{N\to\infty} \dL^N_t(\omega) = \Lambda_t \ge \dLa_t,\quad \mbox{for all } t\in J,
\e*
where $J\subset \R_+$ is the set of all points of continuity of $L$. 
Consequently, $\Lambda=\dLa$ and $\xi(\omega)=\Lc(\dX,\dLa)$. 
In particular, the dominated convergence theorem yields  
\b*
\lim_{N\to\infty} \E[\dL^N_t] = \dLa_t,\quad \mbox{for all } t\in J.
\e*
Set $c:=G(1)$. 
By symmetry, we get $\P(\tau^N_i=\infty)=\P(\tau^N_1=\infty)$ for all $1\le i\le N$. 
Thus,
\b* 
\E\left[S_{\infty}^N\right] = \frac{1}{N} \sum_{i=1}^N \P(\tau^N_i=\infty)= \P(\tau^N_1=\infty).
\e*
Consider the two coupled processes below: for all $t\ge 0$, 
\b*
 X^{N,1}_t = Z^{N,1}+\beta t +B^1_t - G(L^N_t) &\mbox{and}&
  \hat X_t =(Z^{N,1}-\alpha(\gamma_N))+\beta t +B^1_t - G(\Lambda_t) .
\e*
Set $\hat\tau:=\inf\{t\ge 0: \hat X_t\le 0\}$. It suffices to show 
$\lim_{N\to\infty}  \P(\tau^N_1=\infty) =  \P(\hat\tau=\infty)=\P(\tau=\infty)$. 
For $z\in\R$, let $p(z):=\P(z+\beta t+B^1_t>0, \, \forall t\ge 0)$. Then we have $\lim_{z\to\infty}p(z)=1$.
For each $t\ge 0$, we get
\b*
&&  \big|\mathbb P(\tau^N_1=\infty)-\mathbb P(\hat\tau=\infty)\big| \\
&& \qquad \le \big|\mathbb P(\tau^N_1>t)-\mathbb P(\tau^N_1=\infty)\big|+\big|\mathbb P(\tau^N_1>t)-\mathbb P(\hat\tau>t)\big|+\big|\mathbb P(\hat\tau>t)-\mathbb P(\hat\tau=\infty)\big| \\
&& \qquad =  \big|\mathbb P(\tau^N_1>t)-\mathbb P(\tau^N_1=\infty)\big|+\big|\E[\dL_t^N]-\dLa_t\big|+\big|\mathbb P(\hat\tau>t)-\mathbb P(\hat\tau=\infty)\big|,
\e*
where $\P(\tau^N_1>t)=1-\E[\dL_t^N]$ is by symmetry. 
Furthermore,
\b*
\big|\mathbb P(\tau^N_1>t)-\mathbb P(\tau^N_1=\infty) \big| &=& \E[\mathds 1_{\{\tau^N_1>t\}}]- \E[\mathds 1_{\{\tau^N_1>t\}}\mathds 1_{\{\tau^N_1=\infty\}}] \\
&=& \E[\mathds 1_{\{\tau^N_1>t\}}]- \E\big[\P(\tau^N_1=\infty|\tau^N_1>t)\mathds 1_{\{\tau^N_1>t\}}\big].
\e*
On the event $\{\tau^N_1>t\}$, we have 
\b*
\{\tau^N_1=\infty\} &=& \{X^{N,1}_t+\beta(s-t)+(B^1_s-B^1_t)- (G(L^N_s)-G(L^N_t))>0,\forall s\ge t \} \\
&\supset &  \{X_t^{N,1}-c+\beta(s-t)+(B^1_s-B^1_t)>0,\forall s\ge t \},
\e*
which yields by the Markov property $\P(\tau^N_1=\infty|\tau^N_1>t)\ge p(X^{N,1}_t-c)$, and thus
\b*
\big|\mathbb P(\tau^N_1>t)-\mathbb P(\tau^N_1=\infty) \big| &=& \E\big[\mathds 1_{\{\tau^N_1>t\}}\big(1-\P(\tau^N_1=\infty|\tau^N_1>t)\big)\big] \\
&\le&  \E\big[\mathds 1_{\{\tau^N_1>t\}}\big(1-p(X^{N,1}_t-c)\big)\big].
\e*
Set $Y_t:=Z^{N,1}+\beta t+B^1_t$ for all $t\ge 0$ and $\sigma:=\inf\{t\ge 0: Y_t\le 0\}$. Then $X^{N,1}_t\ge Y_t-c$ and $\tau^N_1\le \sigma$. Hence, 
\b*
\big|\mathbb P(\tau^N_1>t)-\mathbb P(\tau^N_1=\infty)  \big|\le \E\big[\mathds 1_{\{\tau^N_1>t\}}\big(1-p(X^{N,1}_t-c)\big)\big] \le \big| \E\big[\mathds 1_{\{\sigma>t\}}\big(1-p(Y_t-2c)\big)\big].
\e*
Similarly, we can establish 
$\big|\mathbb P(\hat\tau>t)-\mathbb P(\hat\tau=\infty)\big| \le \big| \E\big[\mathds 1_{\{\sigma>t\}}\big(1-p(Y_t-\alpha(\gamma_N)-2c)\big)\big]$,
and thus
\b*
 \big|\mathbb P(\tau_1=\infty)-\mathbb P(\hat\tau=\infty)\big|  \le  2\E\big[\mathds 1_{\{\sigma>t\}}\big(1-p(Y_t-\alpha(1)-2c)\big)\big] + \big|\E[\dL^N_t]-\dLa_t\big|.
\e*
As $\lim_{t\to\infty}Y_t=\infty$ holds almost surely, we have by the dominated convergence theorem,
\b*
\lim_{t\to\infty}\E\big[\mathds 1_{\{\sigma>t\}}\big(1-p(Y_t-\alpha(1)-2c)\big)\big] = 0.
\e* 
For any $\eps>0$, there exists $t_{\eps}$ such that
$\E\big[\mathds 1_{\{\sigma>t\}}\big(1-p(Y(t)-\alpha(1)-2c)\big)\big] \le \eps$ for all $t \ge  t_{\eps}$.
Fix an arbitrary $t\in J$ with $t>t_{\eps}$. Then we get
\b*
\lim_{N\to\infty}  \big|\mathbb P(\tau^N_1=\infty)-\mathbb P(\hat\tau=\infty)\big| \le  2\eps + \lim_{N\to\infty}  \big|\E[\dL_t^N]-\dLa_t\big|
=2\eps,
\e*
which yields $\lim_{N\to\infty}\E[\dS_{\infty}^N]=1-\dLa_{\infty}$. Take a sequence $(t_m)_{m\ge 1}\subset J$ satisfying $\lim_{m\to\infty }t_m=\infty$. Then
\b*
\dS_{\infty}^N=1-\dL^N_{\infty} \le 1- \dL^N_{t_m} &\Longrightarrow&  \limsup_{N\to\infty} \dS_{\infty}^N \le 1-\dLa_{t_m}.
\e*
Letting $m\to\infty$, we deduce that $\limsup_{N\to\infty}\dS_{\infty}^N \le 1-\dLa_{\infty}$.  
On the other hand, we get by Fatou's lemma that
$1-\dLa_\infty=\lim_{N\to\infty}\mathbb E[\dS_{\infty}^N] \le \mathbb E[\limsup_{N\to\infty} \dS_{\infty}^N]$, 
which implies $\limsup_{N\to\infty}\dS_{\infty}^N=1-\dLa_{\infty}$. 
We conclude by applying Fatou's lemma,
\b*
\limsup_{N\to\infty}\E[|\dS_{\infty}^N-1+\dLa_{\infty}|] \le \E\left[\limsup_{N\to\infty}|\dS_{\infty}^N-1+\dLa_{\infty}|\right] = 0,
\e*
which yields the desired result.
\end{proof}

\subsection{Proof of Theorem \ref{thm:reg}}
\label{sc42}

\quad We aim to show that the regularized problem \eqref{dyn_reg} is well-posed. 
Define the operators $\Gamma^{\eps}: \Dc\to \Dc$ as follows: for $f\in D$, let $\Gamma^{\eps}[f]\in \Dc$ be defined by
\b* 
\Gamma^{\eps}[f]_t:= 1- \E\left[\exp\left(-\frac{1}{\eps}\int_0^t (X^{f}_s)^-\d s\right)\right],
\e* 
where
\b* 
X^f_t := Z + \beta t + W_t - G(f_t)  &\mbox{and}& \tau^f :=\inf\{t\ge 0:~ X^f_t\le 0\}. 
\e*  
It is easy to see that $\Gamma^{\eps}$ takes values in $\Cc\cap\M$. By definition, $\ell^{\eps}\in \Cc\cap \M$ is a solution to \eqref{dyn_reg} if and only if $\ell^{\eps}$ is a fixed point of $\Gamma^{\eps}$. 
We start with a few technical results.

\begin{lemma} \label{lem:regularity}
Let $f\in \Dc$ be non-decreasing and set
$$A_t:=\left\{\inf_{0\le s\le t}\big(W_s-f_s\big)=0\right\}.$$
Then $A_t$ is negligible for all $t>0$.
\end{lemma}
\begin{proof}
Denote by $\tau:=\inf\{s\ge 0: W_s \le f_s\}$ the hitting time. Then it holds that 
\b* 
A_t= \Big(A_t\cap \{\tau<t\}\Big) \bigcup \Big(A_t\cap \{\tau=t\}\Big) = :B_t\cup C_t.
\e* 
Clearly $C_t\subset \{W_t=f_t\}$. 
As for $B_t$, we have
\b*
B_t &\subset& \{\tau<t \mbox{ and } W_{\tau}=f_{\tau} \mbox{ and } W_s \ge f_s \mbox { for } s\in [\tau, t] \} \\
&\subset& \{\tau<t \mbox{ and } W_{\tau}=f_{\tau} \mbox{ and } W_s \ge f_{\tau} \mbox { for } s\in [\tau, t] \}:=B_t',
\e* 
where the second inclusion holds as $f$ is non-decreasing. 
We conclude by the fact that there are infinite times $u\in [\tau, t]$ so that $W_u<W_{\tau}=f_{\tau}$.  
\end{proof}

\begin{lemma}\label{lem:continuity}
Let $(f^n)_{n\ge 1}\subset \Dc\cap\M$ be a sequence such that $n\mapsto f^n_t$ is non-decreasing for all $t\ge 0$. Denote by $f$ its pointwise limit. For every $x>0$, 
\b*
\lim_{n\to\infty}\mathbb P\left(\inf_{0\le s\le t}(x+W_s-f^n_s)\le 0\right) = \mathbb P\left(\inf_{0\le s\le t}(x+W_s-f_s)\le 0\right),
\e*
holds for all $t\ge 0$.  
\end{lemma}
\begin{proof}
Denote  by $\tau_n$ (resp. $\tau$) the first hitting time of $W-f^n$ (resp. $W-f$) at $-x$. Then 
\b* 
\mathbb P\left(\inf_{0\le s\le t}(W_s-f^n_s)\le -x\right)=\mathbb P(\tau_n\le t) &\mbox{and}& \mathbb P\left(\inf_{0\le s\le t}(W_s-f_s)\le -x\right)=\mathbb P(\tau\le t).
\e* 
By assumption, we have $\tau\le \tau_n$, and thus
\b* 
0\le \mathbb P(\tau\le t)-\mathbb P(\tau_n\le t) =\mathbb P(\tau\le t, \tau_n>t) = \int_{(0,t]}\mathbb P(\tau_n>t|\tau=s)\mathbb P(\tau\in \d s),
\e* 
where the last equality holds as $\tau>0$. Restricted on the set $\{\tau=s\}$ for $s\in (0, t]$, we get
\b*
\{\tau_n>t\}\subset \{W_s-f_s^n> -x\}\cap \{W_s-f_s\le -x\}.
\e*
This implies that
\b* 
\mathbb P(\tau_n>t|\tau=s) \le \mathbb P(f_s-x<W_s \le f_s^n-x)\to 0 \quad \mbox{as } n\to\infty,
\e* 
and we conclude by the dominated convergence theorem.
\end{proof}

Now we prove Theorem \ref{thm:reg}.

\begin{proof}[Proof of Theorem \ref{thm:reg}]
\rmi If the existence holds, then any solution is a strong solution. 
Let $X^{\eps}, Y^{\eps}$ be two arbitrary solutions of \eqref{dyn_reg}. 
Observe that $|X^{\eps}_t-Y^{\eps}_t|$ is deterministic and satisfies 
\b* 
|X^{\eps}_t-Y^{\eps}_t| \le C\mathbb E\left[\frac{1}{\eps}\int_0^t |(X^{\eps}_s)^--(Y^{\eps}_s)^-|\d s\right] \le \frac{C}{\eps}\int_0^t |X^{\eps}_s-Y^{\eps}_s|\d s,\quad \mbox{for all } t\ge 0,
\e* 
where $C>0$ denotes the Lipschitz constant of $G$. This
 yields the uniqueness by Gronwall's inequality. To show the existence, we adopt the proof of the fixed-point theorem. 
 Note that the operator $\Gamma^{\eps}$ preserves the monotonicity, i.e. $f\preceq g \Longrightarrow \Gamma^{\eps}[f]\preceq \Gamma^{\eps}[g]$. 
 Denote $\ell^0\equiv 0$ and set for $n\ge 1$, $\ell^n:=\Gamma^{\eps}[\ell^{n-1}]\in \Cc\cap \M$. 
 Then we get by induction $\ell^{n-1}\preceq  \ell^{n}$ for all $n\ge 1$. 
 Hence, $n\mapsto \ell^n_t$ is non-increasing for all $t\ge 0$, and the pointwise $\ell_t:=\lim_{n\to\infty}\ell^n_t\in [0,1]$ exists. 
 Finally, rewriting the equality
\b* 
\ell^{n+1}_t=G\left(1- \mathbb E\left[\exp\left(-\frac{1}{\eps}\int_0^t\big(Z+\beta s+ W_s-\alpha \ell^n_s\big)^-\d s\right)\right]\right),\quad \mbox{for all } t\ge 0, 
\e* 
 we conclude the desired existence by the dominated convergence theorem. 
 
\medskip

\noindent \rmii  Fix arbitrary $\delta<\eps$. We claim that $\ell^{\delta} \preceq  \ell^{\eps}$. Indeed, it is known from the proof above that
\b* 
\ell^{\delta} = \lim_{n\to\infty} \Gamma^{\delta,n}[\ell^0] &\mbox{and}& \ell^{\eps} = \lim_{n\to\infty} \Gamma^{\eps,n}n[\ell^0], 
\e* 
where $\Gamma^{\delta,n}$ (resp. $\Gamma^{\eps,n}$) is the $n-$composition of $\Gamma^{\delta}$ (resp. $\Gamma^{\eps}$). By definition, we have $\Gamma^{\eps}(f)\preceq  \Gamma^{\delta}(f)$ for all $f\in \Dc$. In particular, $\Gamma^{\eps}(\ell^0)\preceq \Gamma^{\delta}(\ell^0)$. 
Furthermore, we have
\b* 
\Gamma^{\eps,n+1}[\ell^0] = \Gamma^{\eps}\big[\Gamma^{\eps,n}[\ell^0]\big] \preceq \Gamma^{\eps}\big[\Gamma^{\delta,n}[\ell^0]\big]\preceq \Gamma^{\delta}\big[\Gamma^{\delta,n}[\ell^0]\big]=\Gamma^{\delta,n+1}[\ell^0], 
\e* 
where the second inequality follows from the induction and the monotonicity of $\Gamma^{\eps}$. 
Therefore, we get $\Gamma^{\eps,n}[\ell^0]\preceq \Gamma^{\delta,n}[\ell^0]$ for every $n\ge 1$, and $\ell^{\eps}\preceq \ell^{\delta}$. 
Namely, for each $t\ge 0$, $[0,1]\ni \eps \mapsto \ell^{\eps}_t\in [0,1]$ is non-increasing, and we can define the pointwise limit $\hat \ell_t:=\lim_{\eps\to 0+}\ell^{\eps}_t$. 
Thanks to Proposition \ref{prop:conv}, $\hat\ell$ is a solution to \eqref{dyn:lim} .

\begin{proposition}\label{prop:conv}
$\hat \ell$ is a solution to the McKean-Vlasov equation \eqref{dyn:lim}. 
\end{proposition}
\begin{proof}
Set $\hat X_t := Z+\beta t +W_t -\alpha\hat \ell_t$ for $t\ge 0$. 
By definition, we have
\b* 
X^{\eps}_t = Z + \beta t + W_t - G\left(1- \E\left[\exp\left(-\frac{1}{\eps}\int_0^t (X^{\eps}_s)^-\d s\right)\right]\right)
\e*
and 
\b* 
&&\left| X^{\eps}_t - \left(Z + \beta t + W_t - G\left( \P(\inf_{0\le s\le t}X^{\eps}_s\le 0) \right)\right) \right| \\
&= & \left|  G\left(1- \E\left[\exp\left(-\frac{1}{\eps}\int_0^t (X^{\eps}_s)^-\d s\right)\right]\right) - G\left( \E[\mathds{1}_{\{\inf_{0\le s\le t}X^{\eps}_s\le 0\}}]\right) \right|\\
 &\le& C\left|  1- \E\left[\exp\left(-\frac{1}{\eps}\int_0^t (X^{\eps}_s)^-\d s\right)\right] -  \E[\mathds{1}_{\{\inf_{0\le s\le t}X^{\eps}_s\le 0\}}] \right|\\
  &=& C\left| \E\left[\mathds{1}_{\{\inf_{0\le s\le t}X^{\eps}_s> 0\}}- \exp\left(-\frac{1}{\eps}\int_0^t (X^{\eps}_s)^-\d s\right)\right]  \right|\\
  &=&C\E\left[\exp\left(-\frac{1}{\eps}\int_0^t (X^{\eps}_s)^-\d s\right)\mathds{1}_{\{\inf_{0\le s\le t}X^{\eps}_s< 0\}}\right],
\e* 
where the last equality follows from Lemma \ref{lem:regularity}. Applying the dominated convergence theorem, we get
\b* 
\lim_{\eps\to 0+}\E\left[\exp\left(-\frac{1}{\eps}\int_0^t (X^{\eps}_s)^-\d s\right)\mathds{1}_{\{\inf_{0\le s\le t}X^{\eps}_s< 0\}}\right]=0
\e* 
and thus $\hat X_t = Z + \beta t + W_t - \alpha \lim_{\eps\to 0+}\P[\inf_{0\le s\le t}X^{\eps}_s\le 0]$, which yields  the desired result by Lemma \ref{lem:continuity}. 
\end{proof}

Recall that $\dLa$ is the minimal solution to \eqref{dyn:lim}. Then it suffices to show $\ell^{\eps}\preceq \dLa$ for every $\eps>0$. Recall that 
$\dLa=\lim_{n\to\infty}\Gamma_n[\ell^0]$.  
Again, we prove $\Gamma^{\eps}[f]\preceq \Gamma[f]$ for all $f\in \Dc$. Indeed, we rewrite 
\b* 
\Gamma^{\eps}[f]_t =  G\left(\E\left[1- \exp\left(-\frac{1}{\eps}\int_0^t (X^{f}_s)^-\d s\right)\right]\right) &\mbox{and}& \Gamma[f]_t =G\left(\E[\mathds{1}_{\{\inf_{0\le s\le t}X^f_s\le 0\}}]\right)
\e* 
and then compare the two random variables by distinguishing the following cases:
\begin{itemize}
\item On $\{\inf_{0\le s\le t}X^f_s\le  0\}$, it holds 
\b* 
1- \exp\left(-\frac{1}{\eps}\int_0^t (X^{f}_s)^-\d s\right) \le 1 = \mathds{1}_{\{\inf_{0\le s\le t}X^f_s\le 0\}};
\e* 
    \item On $\{\inf_{0\le s\le t}X^f_s>0\}$, it holds
    \b* 
1- \exp\left(-\frac{1}{\eps}\int_0^t (X^{f}_s)^-\d s\right) =0 = \mathds{1}_{\{\inf_{0\le s\le t}X^f_s\le 0\}}.
\e* 
\end{itemize}
This yields in particular $\Gamma^{\eps}[\ell^0]\preceq \Gamma[\ell^0]$. 
Then we have 
\b* 
\Gamma^{\eps,n+1}[\ell^0] = \Gamma^{\eps}\big[\Gamma^{\eps,n}[\ell^0]\big] \preceq \Gamma^{\eps}\big[\Gamma^{n}[\ell^0]\big]\preceq \Gamma\big[\Gamma^{n}[\ell^0]\big]=\Gamma^{n+1}[\ell^0], 
\e* 
where the second inequality follows from the induction and the monotonicity of $\Gamma^{\eps}$. 
Letting $n\to\infty$ and then $\delta\to 0$, we obtain $\hat\ell \preceq \dLa$, which implies that $\hat\ell=\dLa$ is the minimal solution of \eqref{dyn:lim}.
\end{proof}

\section{Proof of Theorem \ref{thm:allo}}
\label{sc5}

\quad In this section, we analyze the effect of budget control. 
Following the same lines as in the proof of Theorem \ref{thm:existence},
we can show that \eqref{dyn:allo_particle} has a unique minimal solution, which yields its wellposedness.
In what follows, we distinguish the cases $\beta<0$, $\beta=0$ and $\beta>0$. 

\subsection{Case of $\beta<0$}
\label{sc51}
\begin{proof}[Proof of Theorem  \ref{thm:allo} {\rm (i)}]
For each $\phi\in\Phi_N$, let $(X^{\phi, N,1},\ldots, X^{\phi, N,N}, L^{\phi,N})$ be an arbitrary solution to \eqref{dyn:allo_particle}. 
For each $t>0$, we have for $i=1,\ldots, N$,
\b*
\lbrace \tau^{\phi}_i > t \rbrace = \left\lbrace \tau^{\phi}_i > t, \, \int_0^t \phi^i_sds > -\frac {\beta t} 2 \right\rbrace  \bigcup \left\lbrace \tau^{\phi}_i > t, \,\int_0^t \phi^i_sds \le -\frac {\beta t} 2 \right\rbrace.
\e*
Hence, 
\b*
\sum_{i=1}^N {\mathds 1}_{\lbrace \tau^{\phi}_i > t \rbrace} &\le& \sum_{i=1}^N {\mathds 1}_{ \lbrace \tau^{\phi}_i > t, \, \int_0^t \phi^i_sds > -\frac {\beta t} 2 \rbrace }  + \sum_{i=1}^N {\mathds 1}_ {\lbrace \tau^{\phi}_i > t, \, \int_0^t \phi^i_sds \le- \frac {\beta t} 2 \rbrace}  \\
&\le&  -\frac{2}{\beta} + \sum_{i=1}^N {\mathds 1}_ {\lbrace \tau^{\phi}_i > t, \, \int_0^t \phi^i_sds \le -\frac {\beta t} 2 \rbrace}. 
\e*
On the event $\lbrace \tau^{\phi}_i > t, \, \int_0^t \phi^i_sds \le -\beta t/2\rbrace$, we get
\b*
X^{\phi,i}_t \le Z^i+\frac{\beta t}{2} +B^i_t,\quad \mbox{for all } t\ge 0,
\e*
and thus
\b*
\left\lbrace \tau^{\phi}_i > t,\int_0^t \phi^i_sds \le - \frac {\beta t} 2\right \rbrace  \subset \left\lbrace Z^i +B^i_t > -\frac {\beta t } 2 \right\rbrace.
\e*
Consequently,
\b*
\E[S^{\phi,N}_\infty]=\lim_{t\to\infty}\E\left[\sum_{i=1}^N {\mathds 1}_{\lbrace \tau^{\phi}_i > t \rbrace} \right] \le -\frac{2}{\beta} +\lim_{t\to\infty} N\P(Z+B_t>-\beta t/2) = -\frac{2}{\beta}.
\e*
\end{proof}

\subsection{Case of $\beta=0$}
\label{sc52}

\quad Theorem \ref{thm:allo} (ii) asserts that at criticality where $\beta = 0$,
the number of surviving banks scales as $\sqrt{N}$
if the initial capital levels $X^{N,i}_0$ are i.i.d. according to $\theta^N$, provided that
$\theta$ is compactly supported.
For simplicity, we assume that $\theta^N=\delta_1$ without loss of generality. 
Since the model \eqref{dyn:allo_particle} is stochastically bounded from above by the Up-the-River model 
via a natural coupling, 
the upper bound in \eqref{eq:infsup0} follows easily from \cite{TT18}.
The main task is to establish the lower bound, 
which boils down to a few lemmas.

Let $\tau^N_{(k)}$ be the first time that $k$ entities go bankrupt;
that is, $k$ members of $X^{N,1}, \ldots X^{N,N}$ enter $(-\infty, 0]$.
The following lemma implies that a significant fraction of banks being ruined
cannot occur in too short time.
\begin{lemma}
\label{lem:shorttime}
Let $\varepsilon > 0$ such that $\alpha \varepsilon < 1$,
and 
\begin{equation}
\label{eq:Taleps}
T_{\alpha, \varepsilon}: = \left(\frac{1 - \alpha \varepsilon}{\widetilde{\mathcal{N}}^{-1} \left(\frac{1}{2}\varepsilon (1 - \varepsilon)^{\frac{1 - \varepsilon}{\varepsilon}} \right)} \right)^2,
\end{equation}
where $\widetilde{\mathcal{N}}(x): = \mathbb{P}(B_1 > x)$ is the tail distribution of standard normal.
Then for $T < T_{\alpha, \varepsilon}$, 
there is $C > 0$ such that
\begin{equation}
\label{eq:taupoly}
\mathbb{P}\left(\tau^N_{(\varepsilon N)} < T\right) \le C N^{-k}, 
\end{equation}
for any $k> 0$, as $N \to \infty$.
\end{lemma}

\begin{proof}
Let $\widetilde{X}^{N,i}_t: = (1 - \alpha \varepsilon) + B^i_t$ be driven by the same Brownian motion as $X^{N,i}$,
and $\widetilde{\tau}^N_{(k)}$ be the $k^{th}$ hitting time of $\{\widetilde{X}^{N,1}, \ldots,\widetilde{X}^{N,N}\}$ to $0$.
By the coupling of $\widetilde{X}^N$ and $X^N$, we get for 
$0 \le t \le \tau^N_{(\varepsilon N)}$,
\begin{equation*}
\frac{1}{N} \sum_{k = 1}^N 1_{\{\tau^N_k \le t\}} \le \varepsilon \quad \mbox{and} \quad
X^{N,i}_t = 1 + B^i_t - \frac{\alpha}{N} \sum_{k = 1}^N 1_{\{\tau^N_k \le t\}} \ge \widetilde{X}^{N,i}_t.
\end{equation*}
Consequently, 
\begin{equation}
\label{eq:couplingtau}
\mathbb{P}\left(\tau^N_{(\varepsilon N)} < T\right) \le \mathbb{P}\left(\widetilde{\tau}^N_{(\varepsilon N)} < T\right).
\end{equation}
We proceed to bounding $\mathbb{P}\left(\widetilde{\tau}^N_{(\varepsilon N)} < T\right)$.
Let $\mathcal{S}_\varepsilon$ be the set of $(\varepsilon N)$-tuples of $\{1, \ldots, N\}$,
so $\# \mathcal{S}_\varepsilon = \binom{N}{\varepsilon N}$
We have
\begin{align}
\label{eq:boundtiltau}
\mathbb{P}\left(\widetilde{\tau}^N_{(\varepsilon N)} < T\right) 
& = \mathbb{P}\left(\bigcup_{S \in \mathcal{S}_\varepsilon} \bigg\{ \max_{k \in S} \widetilde{\tau}^N_k < T \bigg\} \right) \notag \\
& \le \binom{N}{\varepsilon N} \mathbb{P}\left( \max_{k \le \varepsilon N} \widetilde{\tau}^N_k < T \right)
 = \binom{N}{\varepsilon N} \mathbb{P}\left(\widetilde{\tau}^N_1 < T \right)^{\varepsilon N},
\end{align}
where the second inequality is by the union bound, and the third equality follows from the fact that $\widetilde{X}^{N,i}$ are i.i.d.
Note that 
\begin{align*}
\mathbb{P}\left(\widetilde{\tau}^N_1 < T \right) &= \mathbb{P}\left(\inf_{0 \le t \le T} B^i_t \le \alpha \varepsilon - 1 \right)  = \mathbb{P}(|B_T| > 1 - \alpha \varepsilon) = 2 \widetilde{\mathcal{N}}\left(\frac{1 - \alpha \varepsilon}{\sqrt{T}} \right),
\end{align*}
where the second equality is by the reflection principle $\inf_{0 \le t \le T} B^i_t \stackrel{d}{=} -|B^i_T|$ (see \cite[Section 2.8A]{KS91}).
As a result,
\begin{equation}
\label{eq:decoupleprod}
\mathbb{P}\left(\widetilde{\tau}^N_1 < T \right)^{\varepsilon N} =
\exp \left( N \varepsilon \log \bigg(2 \widetilde{\mathcal{N}}\left(\frac{1 - \alpha \varepsilon}{\sqrt{T}} \right)\bigg)\right).
\end{equation}
Moreover,
\begin{equation}
\label{eq:binomasymp}
\binom{N}{\varepsilon N} \sim \exp\left(N (-\varepsilon \log \varepsilon - (1 - \varepsilon) \log(1 - \varepsilon)) \right), \quad \mbox{as } N \to \infty,
\end{equation}
where we only focus on the exponential, and neglect the polynomial factor of $N$ in \eqref{eq:binomasymp}.
Combining \eqref{eq:boundtiltau}, \eqref{eq:decoupleprod} and \eqref{eq:binomasymp} yields
\begin{equation}
\label{eq:asymptiltau}
\mathbb{P}\left(\widetilde{\tau}^N_{(\varepsilon N)} < T\right)
\lesssim \exp\left( N h_{\alpha, \varepsilon}(T) \right),
\end{equation}
where 
\begin{equation}
\label{eq:h}
h_{\alpha, \varepsilon}(T):= \varepsilon \log \bigg(2 \widetilde{\mathcal{N}}\left(\frac{1 - \alpha \varepsilon}{\sqrt{T}} \right)\bigg)- \varepsilon \log \varepsilon - (1 - \varepsilon) \log(1 - \varepsilon).
\end{equation}
Observe that $T \to h_{\alpha, \varepsilon}(T)$ is increasing on $(0, \infty)$, 
and 
\begin{equation*}
\lim_{T \to 0}h_{\alpha, \varepsilon}(T) = -\infty, \quad
\lim_{T \to \infty}h_{\alpha, \varepsilon}(T) = -\varepsilon \log \varepsilon - (1 - \varepsilon) \log(1 - \varepsilon) > 0,
\end{equation*}
so the unique root of $h_{\alpha, \varepsilon}(T) = 0$ on $(0, \infty)$ is $T = T_{\alpha, \varepsilon}$ defined by \eqref{eq:Taleps}.
Therefore, $h_{\alpha, \varepsilon}(T) < 0$ for $T < T_{\alpha, \varepsilon}$, 
and the bound \eqref{eq:taupoly} follows from \eqref{eq:couplingtau} and \eqref{eq:asymptiltau}.
\end{proof}

The next lemma shows that many (most) of the banks surviving up to time  $T_{\alpha, \varepsilon}$ have suitably large capital levels.
\begin{lemma}
\label{lem:numberlow}
Let $\delta > 0$, and let
\begin{equation}
Z_{\alpha, \varepsilon, \delta}:= \#\left\{k: X^{N,k}_{T_{\alpha, \varepsilon}} \ge \alpha + \delta \mbox{ and } X^{N,k}_t > 0 \mbox{ for } t \le T_{\alpha, \varepsilon} \right\},
\end{equation}
be the number of surviving banks whose final capital levels are above $\alpha + \delta$ at time $T_{\alpha, \varepsilon}$.
Let
 \begin{equation}
 \label{eq:paed}
 p_{\alpha, \, \varepsilon, \delta}:= 
 \widetilde{\mathcal{N}}\left( \left(\frac{\alpha + \delta}{1 - \alpha \varepsilon} - 1 \right) \widetilde{\mathcal{N}}^{-1}\left(\frac{1}{2} \varepsilon (1 - \varepsilon)^{\frac{1 - \varepsilon}{\varepsilon}} \right)\right).
 \end{equation}
Then we have
\begin{equation}
\label{eq:eZlb}
\mathbb{E} Z_{\alpha, \varepsilon, \delta} \ge (1 - \varepsilon)p_{\alpha, \varepsilon. \delta} N,
\end{equation}
and for $\lambda < (1 - \varepsilon)p_{\alpha, \varepsilon, \delta}$,
\begin{equation}
\label{eq:pblb}
\mathbb{P}(Z_{\alpha, \varepsilon, \delta} > \lambda N) \ge 1 - C N^{-k},
\end{equation}
for any $k > 0$, as $N \to \infty$.
\end{lemma}

\begin{proof}
Let $\Omega_{\alpha, \varepsilon}:= \left\{\tau^N_{(\varepsilon n)} > T_{\alpha, \varepsilon}\right\}$.
 By Lemma \ref{lem:shorttime}, we get $\mathbb{P}(\Omega^c_{\alpha, \varepsilon}) \lesssim N^{-2}$.
Then we have
\begin{equation}
\label{eq:Zbound}
\begin{aligned}
\mathbb{E}Z_{\alpha, \varepsilon, \delta} 
& =  \mathbb{E}\left(Z_{\alpha, \varepsilon, \delta} 1_{\Omega_{\alpha, \varepsilon}} \right) + \mathbb{E}\left(Z_{\alpha, \varepsilon, \delta} 1_{\Omega^c_{\alpha, \varepsilon}} \right)  \\
&  =\mathbb{E}\left(Z_{\alpha, \varepsilon, \delta} \,|\, \Omega_{\alpha, \varepsilon}\right)\, ( 1- \mathbb{P}(\Omega^c_{\alpha, \varepsilon})) + \mathbb{E}\left(Z_{\alpha, \varepsilon, \delta} 1_{\Omega^c_{\alpha, \varepsilon}} \right)   \\
&  = \mathbb{E}\left(Z_{\alpha, \varepsilon, \delta} \,|\, \Omega_{\alpha, \varepsilon}\right) +  \mathcal{O}(N \cdot N^{-2}) \\
& = \mathbb{E}\left(Z_{\alpha, \varepsilon, \delta} \,|\, \Omega_{\alpha, \varepsilon}\right) + \mathcal{O}(N^{-1}),
\end{aligned}
\end{equation}
where we use the fact that $Z_{\alpha, \varepsilon, \delta} \le N$ and $\mathbb{P}(\Omega^c_{\alpha, \varepsilon}) \lesssim N^{-2}$ in the third equation.
On the event $\Omega_{\alpha, \varepsilon}$, 
there are at least $(1 - \varepsilon) N$ surviving banks,
the capital levels of which are stochastically larger than
Brownian motion starting at $1 - \alpha \varepsilon$ (see \cite{Day83}).
By an obvious coupling (to Brownian particles without control $\phi$),
$(Z_{\alpha, \varepsilon, \delta} \,|\, \Omega_{\alpha, \varepsilon})$ is stochastically larger than $\bin((1 - \varepsilon) N, p_{\alpha, \, \varepsilon, \delta})$,
where
\begin{equation}
\label{eq:comppae}
 p_{\alpha, \, \varepsilon, \delta}:= \mathbb{P}\left(B_{T_{\alpha, \varepsilon}} \ge \alpha(1 + \varepsilon) + \delta - 1 \right)
= \widetilde{\mathcal{N}}\left(\frac{\alpha(1 + \varepsilon) + \delta - 1}{\sqrt{T_{\alpha, \varepsilon}}} \right).
\end{equation}
Recall the definition of $T_{\alpha, \varepsilon}$ from \eqref{eq:Taleps},
and the expression \eqref{eq:comppae} simplifies to the right side of \eqref{eq:paed}.
Combining with \eqref{eq:Zbound} yields the lower bound \eqref{eq:eZlb}.
The bound \eqref{eq:pblb} follows from the stochastic comparison and the Chernorff bound.
\end{proof}
Now we prove Theorem \ref{thm:allo} \rmii.
\begin{proof}[Proof of Theorem \ref{thm:allo} {\rm (ii)}]
The difference between \eqref{dyn:allo_particle} and the Up-the-River model is that 
there is an additional term $-\frac{\alpha}{N} \sum_{k = 1}^N 1_{\{\tau^N_k \le t\}}$,
which drags the capital level down. 
Under the natural coupling of the model \eqref{dyn:allo_particle} and the Up-the-River model,
if a bank survives using a control in the model \eqref{dyn:allo_particle}, 
it will also survive using the same control in the Up-the-River model.
This implies 
\begin{equation*}
\limsup_{N \to \infty} \frac{\dS^{\phi,N}_\infty}{\sqrt{N}} \le \limsup_{N \to \infty} \frac{S^{\tiny \mbox{UR}}_N}{\sqrt{N}},
\end{equation*}
where $S^{\tiny \mbox{UR}}_N$ denotes the maximum number of surviving banks in the Up-the-River model.
By \cite[Theorem 1.2]{TT18}, we have
$\lim_{N \to \infty}  \frac{S^{\tiny \mbox{UR}}_N}{\sqrt{N}} = 4 / \sqrt{\pi}$,
and hence the upper bound in \eqref{eq:infsup0}. 
Now we prove the lower bound.
Fix $m > 0$. We use the following control strategy:
\begin{itemize}
\item
For each bank, exercise no control until it reaches the capital level $N/m$.
\item
If the bank does reach the capital level $N/m$, and is the one of the first $m$ banks achieving so,
then exercise the control $1/m$ thereafter.
\end{itemize}
We write $\dS^{\phi,N}\equiv S_N$ without any danger of confusion.  
For $\lambda < (1 - \varepsilon)p_{\alpha, \varepsilon, \delta}$,
let $\Omega_{\alpha, \varepsilon, \delta, \lambda}: = \{Z_{\alpha, \varepsilon, \delta} > \lambda N\}$.
By Lemma \ref{lem:numberlow}, we have $\mathbb{P}(\Omega_{\alpha, \varepsilon, \delta, \lambda}) \ge 1 - N^{-2}$.
On the event $\Omega_{\alpha, \varepsilon, \delta, \lambda}$, 
there are (at least) $\lambda N$ banks whose capital levels are above $\alpha + \delta$ at time $T_{\alpha, \varepsilon}$.
The capital levels of these banks after $T_{\alpha, \varepsilon}$ are stochastically larger than 
Brownian motion starting at $\delta$.
Recall from \cite[Theorem 7.5.3]{Durrett19},
\begin{equation*}
\mathbb{P}\left(\delta + B_t \mbox{ hits } \frac{N}{m} \mbox{ before } 0\right) = \frac{\delta m}{N}.
\end{equation*}
As a result, the number of banks whose capital levels reach $K/m$ is bounded from below 
by 
$$\min\left(m, \bin(\lambda N, \, \frac{\delta m}{N})\right).$$
Also note that (\cite[Exercise 7.5.2]{Durrett19})
\begin{equation*}
\mathbb{P}\left(\frac{N}{m} - \alpha + B_t + \frac{t}{m} \mbox{ does not hit } 0 \right) = 1 - \exp\left(- \frac{2}{m} \left(\frac{N}{m} - \alpha \right) \right).
\end{equation*}
Therefore,
\begin{align*}
\mathbb{E}S_N & \ge \left[1 -  \exp\left(- \frac{2}{m} \left(\frac{N}{m} - \alpha \right) \right) \right]\mathbb{E}\left[ \min\left(m, \bin \bigg(\lambda N, \, \frac{\delta m}{N} \bigg)\right)\right] + \mathcal{O}(N^{-1}) \\
& \stackrel{m = \theta \sqrt{N}}{=} \min(1, \lambda \delta) \,  \theta \left(1 - \exp\left(-\frac{2}{\theta^2} \right) \right) \sqrt{N} + o(1).
\end{align*}
By taking $\lambda \uparrow (1 - \varepsilon)p_{\alpha, \varepsilon, \delta}$, we get
\begin{equation}
\label{eq:keylb}
\liminf_{N \to \infty} \frac{\mathbb{E} S_N}{\sqrt{N}} \ge 
\min\left(1, (1 - \varepsilon) \delta p_{\alpha, \varepsilon, \delta}\right)   \sup_\theta \left\{ \theta \left( 1 - \exp\left(-\frac{2}{\theta^2} \right) \right)\right\}.
\end{equation}
Note that $\sup_\theta \left\{ \theta \left( 1 - \exp\left(-\frac{2}{\theta^2} \right) \right)\right\} \approx 0.9$, which is attained at $\theta \approx 1.26$.
Further by taking $\varepsilon = \min\left(\frac{1}{2}, \frac{1}{2 \alpha}\right)$ and $\delta = \alpha$ yields 
\begin{equation}
\label{eq:infsup0}
c_\alpha \le \liminf_{N \to \infty} \frac{\mathbb{E}S_N}{\sqrt{N}} \le \limsup_{N \to \infty} \frac{\mathbb{E}S_N}{\sqrt{N}} \le \frac{4}{\sqrt{\pi}},
\end{equation}
where $c_{\alpha} \approx 0.9 \min(1, \rho_\alpha)$ with
\begin{equation}
\rho_\alpha:= \left\{ \begin{array}{ccl}
\left(\alpha - \frac{1}{2} \right) \widetilde{\mathcal{N}}\left( (4 \alpha - 1) \widetilde{\mathcal{N}}^{-1}\left(\frac{1}{4 \alpha} \left(1 - \frac{1}{2 \alpha} \right)^{2 \alpha - 1} \right) \right)  & \mbox{for} & \alpha > 1 \\ 
\frac{1}{2} \alpha \widetilde{\mathcal{N}}\left( \frac{5 \alpha -2}{2 - \alpha} \, \widetilde{\mathcal{N}}^{-1}(1/8) \right) & \mbox{for} & \alpha \le 1.
\end{array}\right.
\end{equation}
\end{proof}

Note that the lower and upper bounds appearing in the proof above may not be optimal. 
It is easy to see from the proof that the lower bound can be numerically improved to
$c'_\alpha \approx 0.9 \min(1, \rho'_\alpha)$ with
\begin{equation*}
\rho'_\alpha:= \sup_{\varepsilon, \delta} \left\{(1 - \varepsilon) \delta p_{\alpha, \varepsilon, \delta} \right\}.
\end{equation*}
But it does not seem that $\rho'_\alpha$ is easily simplified, and has a closed-form expression. 
It is also possible to refine the stochastic comparisons in Lemmas \ref{lem:shorttime} and \ref{lem:numberlow} to further improve
the bounds in \eqref{eq:infsup0}.

\subsection{Case of $\beta>0$}
\label{sc53}

\quad Finally, we consider the positive economy where the capital level of each bank has a positive drift $\beta >0$:
\begin{equation}
\label{eq:positivehit}
X^{N,i}_t = X^{N,i}_0 + \beta t + B_t^i - \frac{\alpha}{N} \sum_{k = 1}^N 1_{\{\tau^N_k \le t\}}.
\end{equation}
The question is whether the budget control is helpful,
and what is the optimal control to maximize the number of surviving banks.
As previously mentioned, 
this problem is challenging. 

To provide insights, we focus on the setting without interaction, i.e. $\alpha=0$. 
In this case, the equations
\begin{equation}
\label{eq:positiveSDE}
X_t^{\phi, N, i}= X^{N,i}_0 + \beta t + B_t^i + \int_0^t \phi_s^{N,i} ds,
\end{equation}
have a unique solution,
and we look for 
\begin{equation*}
S_N = \max_\phi S^{\phi, N} : = \max_{\phi} \sum_{k = 1}^N 1_{\{\tau^{\phi, N}_k = \infty\}}.
\end{equation*}
where $\tau^{\phi, N}_k:=\inf\{t>0: X^{\phi, N,i}_t \le 0\}$.
For $\phi = 0$, we have $X^{0, N,i}_t = X_0^{N,i} + \beta t + B^i_t$, which are i.i.d. Brownian motion with drift $\beta > 0$.
By \cite[Exercise 7.5.2]{Durrett19},
\begin{equation}
\label{eq:Xnocontrol}
\mathbb{P}(X^{N,i}_0 + \beta t + B^i_t \mbox{ does not hit } 0) = 1 - \E \left[e^{-2 \beta X^{N,1}_0}\right].
\end{equation}
Thus, $\mathbb{E}S^{0,N} = \left( 1 - \E \left[e^{-2 \beta X^{N,i}_0}\right] \right) N$. 
We aim to prove that in the supercritical case where $\beta > 0$,
the number of surviving banks scales as $N$,
and moreover, the unit control does not help.
That is, for any $\phi \in \Phi_N$, 
\begin{equation*}
\lim_{N \to \infty} \frac{\mathbb{E}S^{\phi, N}}{N} = \lim_{N \to \infty} \frac{\mathbb{E}S^{0, N}}{N} = 1-\E \left[e^{-2 \beta X^{N,1}_0}\right].
\end{equation*}
To this end, we need the following lemma.

\begin{lemma}
\label{lem:Greenpos}
Let $Y^x_t = x + \beta t + B_t$, and $\Psi(t,x): = \mathbb{P}\left(\inf_{s \le t} Y^x_s > 0\right)$.
Then
\begin{enumerate}
\item
$\Psi(t,x)$ solves the PDE 
\begin{equation}
\label{eq:IBPDE}
\frac{\partial \Psi}{\partial t}  = \beta \frac{\partial \Psi}{\partial x}  + \frac{1}{2} \frac{\partial^2 \Psi}{\partial t^2}, \quad
\Psi(t, 0) = 0, \quad  \Psi(0,x) = 1_{\{x>0\}}.
\end{equation}
\item
We have
\begin{equation}
\label{eq:absorbpos}
\Psi(t,x) = \Phi\left( \frac{x + \beta t}{\sqrt{t}}\right) - e^{-2 \beta x} \Phi\left( \frac{-x + \beta t}{\sqrt{t}}\right),
\end{equation}
where $\Phi(x) = \frac{1}{\sqrt{2 \pi}} \int_{-\infty}^x e^{-\frac{z^2}{2}} dz$ 
is the cumulative distribution function of standard normal.
\end{enumerate}
\end{lemma}

\begin{proof}
(1) Let $T_x: = \inf\{t>0: Y^x_t = 0\}$ be the first time at which $Y^x$ hits $0$. 
By duality, we have
\begin{equation*}
\Psi(t,x):= \mathbb{P}\left(\inf_{s \le t} Y^x_s > 0\right) = \mathbb{P}(T_x > t),
\end{equation*}
which is the Green function of $Y^x$ absorbed at $0$. 
It follows from the classical Feynman–Kac formula (see \cite[Section 5.4]{Gard04}) that 
$\Psi$ solves the initial-boundary problem \eqref{eq:IBPDE}.

(2) It is easy to check that the formula \eqref{eq:absorbpos} satisfies the PDE \eqref{eq:IBPDE}.
It can also be read from \cite{BW76}, \cite[Example 5.5]{CM65}.
\end{proof}

\begin{proof}[Proof of Theorem \ref{thm:allo} {\rm (iii)}]
First by the natural coupling, we have for any control $\phi \in \Phi_N$,
\begin{equation*}
\mathbb{E}S^{\phi, N} \ge \mathbb{E}S^{0, N} = \left(1-\E \left[e^{-2 \beta X^{N,1}_0}\right] \right) N.
\end{equation*}
This implies that $\liminf_{N \to \infty} \mathbb{E}S_N/N \ge 1-\E \left[e^{-2 \beta X^{N,1}_0}\right]$. Recall the definition of $\Psi(t,x)$ in Lemma \ref{lem:Greenpos}.
Let $Z_s = \sum_{i = 1}^N \Psi(t-s, X_s^{\phi, N, i})$, $s \le t$. 
Applying It\^o's formula, we get
\begin{equation}
\label{eq:ItoZ}
\begin{aligned}
Z_t & = N \Psi(t,X^{N,i}_0) 
+ \sum_{i = 1}^N \int_0^t \left(- \frac{\partial}{\partial t} + \beta \frac{\partial}{\partial x} + \frac{1}{2} \frac{\partial^2}{\partial x^2} \right) \Psi(t-s, X_s^{\phi, N,i}) ds \\
& \qquad \qquad \qquad + \int_0^t \sum_{i = 1}^N \phi^{N,i}(s) \frac{\partial}{\partial x}\Psi(t-s, X_s^{\phi, N,i})  ds + \mbox{martingale} \\
& = N \Psi(t,X^{N,i}_0) + \int_0^t \sum_{i = 1}^N \phi^{N,i}(s) \frac{\partial}{\partial x}\Psi(t-s, X_s^{\phi, N,i})  ds + \mbox{martingale},
\end{aligned}
\end{equation}
where the second equality follows from \eqref{eq:IBPDE}. 
Next by \eqref{eq:absorbpos}, we get
\begin{equation}
\label{eq:Psix}
\frac{\partial \Psi}{\partial x}(t,x) = p(t,x + \beta t) + e^{-2 \beta x} p(t,x - \beta t) + 2 \beta e^{-2 \beta x} \Phi\left(\frac{-x + \beta t}{\sqrt{t}} \right),
\end{equation}
where $p(t,x): = \frac{1}{\sqrt{2 \pi t}} \exp\left(- \frac{x^2}{2t} \right)$ is the heat kernel. 
Observe that for $x \ge 0$, 
\begin{equation}
\label{eq:estterms}
p(t, x+ \beta t) \lesssim t^{-\frac{1}{2}} e^{-\beta^2 t}, \quad
e^{-2 \beta x} p(t,x - \beta t) \lesssim t^{-\frac{1}{2}} e^{-\beta^2 t}, \quad
e^{-2 \beta x} \Phi\left(\frac{-x + \beta t}{\sqrt{t}} \right) \le 1.
\end{equation}
By \eqref{eq:Psix} and \eqref{eq:estterms}, we get
\begin{equation}
\label{eq:contphi}
\int_0^t \sum_{i = 1}^N \phi^{N,i}(s) \frac{\partial}{\partial x}\Psi(t-s, X_s^{\phi, N,i})  ds 
\le C \int_0^t \frac{e^{-\beta^2 (t-s)}}{\sqrt{t -s}} ds + 2 \beta t \le C' + 2 \beta t.
\end{equation}
Combining \eqref{eq:ItoZ} and \eqref{eq:contphi} yields
$\mathbb{E}S_N  \le \mathbb{E}Z_t \le N\E \Psi(t,X^{N,1}_0) + 2 \beta t + C'$ for any $t > 0$.
Fixing $\varepsilon > 0$ and taking $t = \frac{\varepsilon}{2 \beta} N$, we get
\begin{equation}
\label{eq:limsupeps}
\limsup_{N \to \infty} \frac{\mathbb{E}S_N}{N} \le \lim_{N \to \infty} \E\Psi\left(\frac{\varepsilon N}{2 \beta}, X^{N,1}_0\right) + \varepsilon
= 1-\E \left[e^{-2 \beta X^{N,1}_0}\right] + \varepsilon.
\end{equation}
Since \eqref{eq:limsupeps} holds for any $\varepsilon > 0$, we conclude that $\limsup_{N \to \infty} \mathbb{E}S_N/N \le 1-\E \left[e^{-2 \beta X^{N,1}_0}\right]$.
\end{proof}

\bibliographystyle{abbrv}
\bibliography{unique} 

\end{document}